\documentclass[12pt]{amsart}
\usepackage[utf8]{inputenc}
\usepackage[foot]{amsaddr}
\usepackage{amsmath,amsfonts,amsthm,amssymb,verbatim,etoolbox,color}
\usepackage{enumitem} 
\usepackage{epigraph}
\usepackage{appendix}
\usepackage{xfrac}
\usepackage{mathtools}
\usepackage{flexisym}
\usepackage{bbm}
\usepackage{float}
\usepackage{graphicx}

\usepackage[margin=1in]{geometry}
\usepackage[
    backend=biber,
    maxnames=4,
    maxalphanames=4,
    style=alphabetic,
    backref
  ]{biblatex}
\usepackage{scalerel}

\usepackage{hyperref}
\hypersetup{
    colorlinks=true,
    linkcolor=red,
    filecolor=magenta,      
    urlcolor=magenta,
    pdftitle={Multiplicative subgroups of prime fields are not sumsets},
    pdfauthor = {Misha Rudnev and Fred Tyrrell},
    pdfpagemode=FullScreen,
    linktocpage = true,
    citecolor = blue,
    }

\usepackage{theoremref}

\newtheorem{thm}{Theorem}
\newtheorem{lemma}[thm]{Lemma}

\newtheorem{corollary}[thm]{Corollary}

\newtheorem{prop}[thm]{Proposition}
\theoremstyle{definition}
\newtheorem{defn}[thm]{Definition}
\newtheorem{remark}[thm]{Remark}

\numberwithin{thm}{section}

\newcommand{\F}{\mathbb{F}}

\newcommand{\sub}{\subseteq}

\newcommand{\bra}[1]{\left(#1\right)}
\newcommand{\sqbra}[1]{\left[#1\right]}

\newcommand{\abs}[1]{\left\lvert {#1} \right \rvert}

\newcommand{\al}{\alpha}

\newcommand{\ep}{\varepsilon}
\newcommand{\be}{\beta}

\newcommand{\ga}{\gamma}

\newcommand{\Bb}{B_b^{-1}}

\newcommand{\Ab}{\frac{1}{A+b}}

\def\bpm{\begin{pmatrix}}
\def\epm{\end{pmatrix}}

\title{Multiplicative subgroups of prime fields are not sumsets}
\author{Misha Rudnev}
\author{Fred Tyrrell}
\email{m.rudnev@bristol.ac.uk}
\email{fred.tyrrell@bristol.ac.uk}
\address{Fry Building, School of Mathematics, University of Bristol}
\date{\today}

\begin{document}
\begin{abstract}
Let $H \leq \mathbb{F}_p^*$ be a proper multiplicative subgroup, and suppose that $H = A+B$ for some $A,B \subseteq \mathbb{F}_p$. We prove that either one of the summands is a singleton, or $|A|=|B|=2$ and $|H|=4$. In particular, no proper multiplicative subgroup of $\mathbb{F}_p^*$ can be written as $A+B$ with $|A|,|B|>2$. 
\smallbreak
Our proof builds on the Hanson-Petridis polynomial method and Kalmynin's subsequent resolution of S\'ark\"ozy's conjecture for quadratic residues. Using Kalmynin's $|A|=|B|$ theorem as a structural input, we develop uniform combinatorial and arithmetic arguments which apply to multiplicative subgroups of arbitrary index.
\end{abstract}
\maketitle

\section{Introduction}
For sets $A,B\sub\F_p$, their sumset is defined as 
\[A+B=\{a+b:a\in A, b\in B\}.\] 
Motivated by Ostmann's inverse Goldbach problem and related questions about additive decompositions of the squares, S\'ark\"ozy \cite{sarkozy2012additive} asked whether the set of quadratic residues
\[\mathcal R_p=\{x^2:x\in\F_p^*\}\]
can admit a non-trivial additive decomposition $\mathcal R_p=A+B$, with $\abs{A},\abs{B}\geq2$.
\smallbreak
A substantial body of work has been devoted to S\'ark\"ozy's original conjecture. Shkredov \cite{shkredov2014sumsets} settled the symmetric case $A=B$, classifying decompositions of the form $\mathcal R_p=A+A$, together with the corresponding restricted-sumset problem. Blackburn, Konyagin and Shparlinski \cite{blackburn2015counting} obtained a non-trivial upper bound on the number of possible decompositions of $\mathcal R_p$, while Chen and Xi \cite{chen2022conjecture} sharpened the known restrictions on the sizes of the summands and on the number of uniquely represented elements.
\smallbreak
Hanson and Petridis \cite{Hanson2021refined} made a breakthrough by developing a Stepanov-type auxiliary-polynomial method for sumsets contained in sets of roots of unity. In the extremal situation arising from an exact decomposition $\mathcal R_p=A+B$, their method produces a polynomial factorisation which imposes strong restrictions on the summands. In particular, they established S\'ark\"ozy's conjecture for almost all primes. Kalmynin \cite{kalmynin2025additive} further developed the Hanson-Petridis method and resolved S\'ark\"ozy's conjecture for every prime. He also proved an important structural result for additive decompositions of arbitrary multiplicative subgroups, which we state below as \thref{kalmyninsize}.
\smallbreak
Shparlinski \cite{shparlinski2013additive} initiated the systematic study of additive decompositions of arbitrary multiplicative subgroups $H\leq\F_p^*$, often referred to as the \emph{generalised S\'ark\"ozy conjecture}. Shparlinski obtained quantitative restrictions on the possible sizes of the summands in such a decomposition. Shkredov \cite{shkredov2020any} used state-of-the-art incidence theory to show that, for every $\ep>0$, a multiplicative subgroup satisfying 
\[1\ll_\ep\abs{H}\leq p^{2/3-\ep}\]
cannot be represented as $A+B$ with $\abs{A},\abs{B}>1$.
\smallbreak
Yip \cite{yip2024additive} used the Hanson-Petridis method to obtain further irreducibility results for several families of large subgroups over general finite fields, together with results concerning decompositions of the form $A+A$ and $A+B+C$.
\smallbreak
Several related variants and multiplicative analogues have also been investigated. Wu and She \cite{wu2023additive} studied additive decompositions of the set of nonzero cubes, while Yip \cite{yip2025restricted} established restricted-sumset analogues of S\'ark\"ozy's conjecture and more general results for restricted sumsets contained in multiplicative subgroups. In a complementary direction, Kim, Yip and Yoo \cite{kim2023diophantine,kim2026multiplicative} studied multiplicative decompositions of nonzero shifts of multiplicative subgroups. Among their results, they rule out representations of the nonzero part of such a shift as a ratio set $A/A$. They note that, although their arguments are motivated by Kalmynin's method, their resulting proofs are simpler. In \thref{rem:yip}, we indicate a connection between part of their argument and the first-order reciprocal relation appearing in Kalmynin's proof of \thref{kalmyninsize}.
\smallbreak
Since the first version of this paper appeared, several closely related preprints have been posted. Yip and Yoo \cite{yip2026multiplicative} obtained a complete classification of multiplicative subgroups which are not restricted sumsets, Yoo \cite{yoo2026multiplicative} studied multiplicative decompositions of shifted multiplicative subgroups, and Cochrane \cite{cochrane2026sumsets} studied when multiplicative subgroups of general finite fields are a generalised arithmetic progression. Yip and Yoo \cite{yip2026additive} subsequently gave another proof of \thref{main,kalmyninsize}, based on differential identities for the root polynomials associated with the two summands.

\smallbreak
Despite the resolution of S\'ark\"ozy's original conjecture and the substantial progress for numerous families of multiplicative subgroups, the binary additive decomposition problem remained open for an arbitrary proper multiplicative subgroup of a prime field. Our main result gives a complete classification of all such decompositions.

\begin{thm}\thlabel{main}
Let $p$ be a prime, and let $H \leq \F_p^*$ be a proper multiplicative subgroup. Suppose that
\[H=A+B\]
for some $A, B \sub \F_p$. Then one of the following holds:
\begin{enumerate}
    \item $\abs{A}=1$ or $\abs{B}=1$,
    \item $\abs{A}=\abs{B}=2$, $\abs{H}=4$ and $p \equiv 1 \pmod 4$.
\end{enumerate}
In particular, no proper multiplicative subgroup $H$ of $\F_p^*$ can be written as $H=A+B$ with $\abs{A},\abs{B}>2$.
\end{thm}

Our starting point is the following theorem of Kalmynin. Building on the work of Hanson and Petridis \cite{Hanson2021refined}, who showed that if $H$ admits an additive decomposition, then each element of $H$ is represented uniquely as a sum in $A+B$, Kalmynin proved that the two summands must have the same size.

\begin{thm}[\cite{kalmynin2025additive}, Theorem 2]\thlabel{kalmyninsize}
Let $p$ be a prime, let $H \leq \F_p^*$ be a proper multiplicative subgroup, and suppose that
\[H=A+B\]
for some $A,B\sub \F_p$ and $\abs{A},\abs{B}\geq 2$. Then
\[\abs{A}=\abs{B}=\sqrt{\abs{H}}.\]
\end{thm}

\begin{remark}\thlabel{trivial}
The trivial decomposition $H = \{x\} + (H-x)$, where one of $A,B$ is a translate of a subgroup and the other is a singleton, is always possible, which is case $(1)$ of \thref{main}. 
\smallbreak
In case $(2)$, if $\abs{A}=\abs{B}=2$, then by \thref{kalmyninsize} we must have $\abs{H}=4$, and so $4 \mid p-1$ by Lagrange's theorem, which is equivalent to $p \equiv 1 \pmod 4$. Since $p \equiv 1 \pmod 4$, there is $i \in \F_p$ such that $i^2=-1$. Thus setting $A = \{0,-1-i\}$ and $B = \{1,i\}$, we have
\[A+B = \{1,-1,i,-i\},\]
which is a multiplicative subgroup of $\F_p^*$, the fourth roots of unity mod $p$. 
\end{remark}

\subsection{Sketch of the proof}
We give a sketch of the proof here, and explain how the paper is organised.
\smallbreak
Suppose that $H = A+B$ with $\abs{A},\abs{B} > 2$. Write $H = \mu_d$, where $\mu_d$ denotes the $d$-th roots of unity mod $p$. By \thref{kalmyninsize} we have
\[\abs{A}=\abs{B}=\al, \qquad d = \al^2.\]
The subsequent argument is self-contained, although some parts are inspired by Kalmynin's proof of S\'ark\"ozy's conjecture for quadratic residues. The essential new difficulty is that, in Kalmynin's setting, one has the special relation $p=2\alpha^2+1$, whereas for an arbitrary multiplicative subgroup we must treat $p=M\alpha^2+1$ for every $M \geq 2$. Our main task is therefore to derive a polynomial congruence which is valid uniformly in $M$ and then to analyse its solutions without using the special assumption $M=2$.
\smallbreak
We begin in Section \ref{section3} with the reciprocal identities which underlie the argument. These correspond to the two relations obtained by Kalmynin in his proof of S\'ark\"ozy's conjecture. We give a somewhat different derivation, starting from a modified form of the Hanson-Petridis polynomial identity. For fixed $b \in B$, this is applied to the reciprocal sets 
\[\frac{1}{A+b}=\left\{\frac{1}{a+b}:a\in A\right\}, \qquad B_b^{-1}=\left\{\frac{1}{b-b'}:b'\in B\setminus\{b\}\right\}\]
to obtain a polynomial identity in which the left-hand side is a sum over $A$, and the right-hand side is a product over $B$, allowing us to transfer symmetric information between these two sets. Comparing the first two non-trivial coefficients gives what we refer to as the \emph{first-} and \emph{second-order} reciprocal transfer identities.
\smallbreak
We next study the least nonzero power sum. After translating the summands in opposite directions, we may assume that $p_1\bra{A}=0$. Let $k$ be the least positive integer such that
\[p_k\bra{A}=\sum_{a\in A}a^k\neq0.\]
Kalmynin introduced the same parameter in his proof of S\'ark\"ozy's conjecture. He showed that $k$ is also the least positive index for which $p_k\bra{B}\neq0$, that 
\[p_k\bra{A}=-p_k\bra{B},\]
and that $k$ is even. In Section \ref{section4}, we give elementary proofs of these facts in our setting. The first two assertions follow directly from the uniqueness of representations in $A+B=\mu_d$, while the parity of $k$ follows from the first-order reciprocal transfer identity and a symmetrisation argument.

\smallbreak
The main algebraic calculation is carried out in Section \ref{section5}. We introduce several sums involving two reciprocal factors and evaluate them using partial fractions, complete homogeneous symmetric polynomials, and symmetric and antisymmetric combinations. Combining these identities with the second-order reciprocal transfer identity gives the polynomial congruence
\[P\bra{\alpha,k}\equiv0\pmod p,\] 
where 
\[P\bra{\alpha,x} =3\bra{x^2-3x-2}\alpha^2 +4\bra{x^2+2}\alpha -\bra{x^2-3x+2}.\] 
The derivation is entirely algebraic and combinatorial and is valid for multiplicative subgroups of arbitrary index. In the special case of quadratic residues, where $p=2\alpha^2+1$, this congruence reduces to the polynomial relation obtained by Kalmynin in his proof of S{\'a}rk{\"o}zy's conjecture. His derivation uses differential forms and residue calculations, whereas our symmetric and antisymmetric framework produces the uniform relation directly.
\smallbreak
In Section \ref{section6}, we prove that $k\mid\alpha$. The first and last steps of this argument also appear in Kalmynin's treatment of quadratic residues. Assume for a contradiction that there is some $m\leq\alpha$ such that $k\nmid m$ and $p_m\bra{A}\neq0$, and choose the least such $m$. Its minimality implies that every product of lower power sums of total degree $m$ vanishes, and an argument analogous to the proof for $k$ shows that $m$ is even. We may therefore repeat the calculation of Section \ref{section5} with $m$ in place of $k$, obtaining 
\[P\bra{\alpha,m}\equiv0\pmod p.\] 
The main new ingredient in this section is a root-uniqueness result for the uniform polynomial $P$. We prove that, whenever $p=M\alpha^2+1$ is prime, the congruence 
\[P\bra{\alpha,x}\equiv0\pmod p\] 
has at most one solution with $1\leq x\leq\alpha$. Since $k$ is already such a solution, the integer $m$ cannot exist. It follows that 
\[p_i\bra{A}=p_i\bra{B}=0\] 
for every $1\leq i\leq\alpha$ with $k\nmid i$. Finally, choosing $S\in\{A,B\}$ such that $0\notin S$ and applying Newton's identities, we obtain $e_i\bra{S}=0$ whenever $k\nmid i$. Since 
\[e_\alpha\bra{S}=\prod_{s\in S}s\neq0,\] 
we conclude that $k\mid\alpha$.
\smallbreak
The proof is completed in Section \ref{section7} by ruling out the congruence 
\[P\bra{\alpha,k}\equiv0\pmod p\] 
under the conditions already established, namely that $k$ is even and $k\mid\alpha$. Writing 
\[k=2K,\qquad \alpha=2Kn,\qquad p=M\alpha^2+1,\] 
we combine the assumption $p\mid P\bra{\alpha,k}$ with the identity $p=M\al^2+1$. After reducing the resulting expression modulo $\alpha^2$, we obtain an auxiliary positive integer $\ell$ and show that one of the following alternatives must hold: 
\[\ell n=4K, \qquad \text{or} \qquad n=1,\qquad \ell=4K-1.\] 
Elementary divisibility arguments and short finite checks reduce both alternatives to a finite collection of possibilities, all of which are excluded. Consequently, 
\[P\bra{\alpha,k}\not\equiv0\pmod p,\] contradicting the congruence obtained in Section \ref{section5} and completing the proof of \thref{main}.

\section*{Acknowledgements}
We thank Ben Hobson for interesting discussions during the early stages of this project while a Martingale scholar at Bristol. We also thank Bogdan Nica for informing us that \thref{homo} is a folklore identity of Sylvester.
\smallbreak
ChatGPT, accessed via a ChatGPT Plus subscription and using the GPT-5.4, GPT-5.5 and GPT-5.6 models, was used in the preparation of this manuscript for algebraic and numerical checks, generation of code, assistance with presentation, and proofreading. All mathematical content is the sole responsibility of the authors.
\smallbreak
The second author is funded by the Heilbronn Institute for Mathematical Research.

\section{Preliminaries}\label{section2}
In this section, we introduce the preliminary definitions and facts necessary to discuss the \emph{Hanson-Petridis} polynomial introduced in \cite{Hanson2021refined}, which forms the basis of the subsequent work of Kalmynin \cite{kalmynin2025additive}. We also record other necessary preliminary results, and fix notation. 
\smallbreak
Throughout the paper, we let $A, B \sub \F_p$ with $\abs{A}=\al$ and $\abs{B}=\be$. We denote the $d$-th roots of unity mod $p$ by $\mu_d=\{x \in \F_p:x^d=1\}$. We write $\F_p^*$ to denote the multiplicative group of $\F_p$, which is a cyclic group of order $p-1$. If $H$ is a subgroup of $\F_p^*$ then in fact $H = \mu_d$, the $d$-th roots of unity mod $p$, for some $d \mid p-1$.
\smallbreak
We also define the homogeneous and elementary polynomials, power sums, and record some important identities.
\begin{defn}\thlabel{homoelem}
    For a set $S$ of size $\al$ and $m \geq 0$, we define the degree $m$ complete homogeneous polynomial over $S = \{s_1,\ldots,s_{\al}\}$ as
        \[h_m\bra{S} := \sum\limits_{1 \leq i_1 \leq \cdots \leq i_m \leq \al}s_{i_1}\cdots s_{i_m}.\]
        Similarly, we define the degree $m$ elementary symmetric polynomial over $S$ as
            \[e_m\bra{S} := \sum\limits_{1 \leq i_1 < \cdots < i_m \leq \al}s_{i_1}\cdots s_{i_m}.\]
            Finally, we define the $m$-th power sum of $S$ as
            \[p_m\bra{S}:=\sum\limits_{s \in S}s^m.\]
\end{defn}
We will use the following standard identities about symmetric and elementary polynomials in relation to the power sums, known as Newton's identities.
\begin{lemma}\thlabel{newton}
For a finite set $S$,
\begin{align*}
h_1(S) &= p_1(S),\\
h_2\bra{S} &=\frac{1}{2}(p_1(S)^2 + p_2(S)),\\
h_3(S) &= \frac{1}{6} (p_1(S)^3 + 3p_1(S)p_2(S) + 2p_3(S)), \\
e_1(S) &= p_1(S),\\
e_2(S) &= \frac{1}{2}(p_1(S)^2 - p_2(S)),\\
e_3(S) &= \frac{1}{6}(p_1(S)^3 - 3p_1(S)p_2(S) + 2p_3(S)).  
\end{align*}
More generally, for any $1 \leq j \leq \abs{S}$,
\[j e_j\bra{S}=\sum_{i=1}^j (-1)^{i-1}e_{j-i}\bra{S}p_i(S)\]
\end{lemma}
\begin{remark}
For a set $S \sub \F_p$, knowing the values of $e_1\bra{S},\ldots,e_{\abs{S}}\bra{S}$ completely determines $S$, as they are precisely the coefficients of the characteristic polynomial for $S$,
\[\prod_{s \in S}\bra{X-s}.\]
If $\abs{S} < p$, then by \thref{newton} the values of $p_1\bra{S},\ldots,p_{\abs{S}}\bra{S}$ also determine $S$.
\end{remark}

\section{Reciprocal transfer identities}\label{section3}

In this section, we prove a modified version of the Hanson-Petridis identity in \cite{Hanson2021refined}, and then use this identity to study reciprocal translates of the form $\frac{1}{A+b}$. First, we need to introduce the Hanson-Petridis coefficients.

\begin{defn}\thlabel{coeffs}
For $S \sub \F_p$ and $s \in S$, we define the coefficient $c_S(s)$ by
\[\sum\limits_{s \in S}s^jc_S(s) = 
\begin{cases}0, & 0 \leq j \leq \abs{S}-2,\\
1, & j=\abs{S}-1.
\end{cases}\]
\end{defn}
We can write \thref{coeffs} using a Vandermonde matrix as follows,
\[\begin{pmatrix}
1 & 1 & \cdots & 1\\
s_1 & s_2 & \cdots & s_{\abs{S}} \\
\vdots & \vdots & \ddots & \vdots\\
s_1^{\abs{S}-1} & s_2^{\abs{S}-1} & \cdots & s_{\abs{S}}^{\abs{S}-1}
\end{pmatrix}
\begin{pmatrix}
    c_S\bra{s_1}\\
    c_S\bra{s_2}\\
    \vdots \\
    c_S\bra{s_{\abs{S}}}
\end{pmatrix}
= \begin{pmatrix}
    0 \\
    \vdots\\
    0\\
    1
\end{pmatrix}.
\]
We have the following explicit formula for the Hanson-Petridis coefficients in \thref{coeffs}. 
\begin{lemma}\thlabel{coeffs2}
Let $S\sub\F_p$ be nonempty, and let $s\in S$. Then
 \[c_S(s) = \frac{1}{\prod\limits_{s' \neq s}\bra{s-s'}}.\]   
\end{lemma}

We also note the following fact about the coefficients $c_S\bra{s}$.
\begin{lemma}\thlabel{homo}
    For $m \geq 0$, and any $S \subsetneq \F_p$,
    \[\sum\limits_{s \in S}c_S(s)s^{m+\abs{S}-1}=h_m(S).\]

\end{lemma}

Kalmynin gives a proof of \thref{coeffs2,homo} via generating functions as Lemmas 1 and 2 in \cite{kalmynin2025additive}. We give a variant of the proof of both lemmas in the appendix via combinatorial identities using partial fractions.
\smallbreak
We can now prove the modified Hanson-Petridis identity.
\begin{lemma}\thlabel{modifiedHP}
    Let $S \sub \mu_d$ be a set of size $\al$. For $s \in S$, define
    \[k_S\bra{s}=\frac{c_S\bra{s}}{s}, \qquad C_S=\sum_{s \in S}k_S\bra{s}.\]
    Let $T \sub \F_p^*$ be a set of size $\ga$, such that
    \[S+t \sub t \mu_d\]
    for every $t \in T$. If $d = \al\bra{\ga+1}$, then
    \[\sum_{s \in S}k_S\bra{s}\bra{x+s}^{\al+d-1} - C_S x^{\al + d-1}=\binom{\al+d-1}{\al}x^{\al-1}\prod_{t \in T}\bra{x-t}^{\al}.\]
\end{lemma}
\begin{proof}
    First, we evaluate
    \[C_S = \sum_{s \in S}k_S\bra{s}=\sum_{s \in S}\frac{c_S\bra{s}}{s}.\]
    By the previous remark, we have
    \[c_S(s) = \frac{1}{\prod\limits_{s' \neq s}\bra{s-s'}},\]
    and hence
    \[C_S = \sum_{s \in S}\frac{1}{s}\frac{1}{\prod\limits_{s' \neq s}\bra{s-s'}}.\]
    We can rewrite this as
    \[ \sum_{s \in S}\frac{1}{s}\frac{1}{\prod\limits_{s' \neq s}\sqbra{\bra{\frac{1}{s'}-\frac{1}{s}}ss'}}\]
    \[=\sum_{s \in S}\frac{1}{s^{\al}\prod\limits_{s' \neq s}\sqbra{s'\bra{\frac{1}{s'}-\frac{1}{s}}}}\]
    \[=\bra{\prod\limits_{s \in S}\frac{1}{s}}\sum_{s \in S}\frac{1}{s^{\al-1}\prod\limits_{s' \neq s}\bra{\frac{1}{s'}-\frac{1}{s}}}.\]
But writing $\frac{1}{S}=\{\frac{1}{s}:s \in S\}$, we have
\[c_{\frac{1}{S}}\bra{\frac{1}{s}}=\frac{1}{\prod\limits_{\frac{1}{s'}\neq \frac{1}{s}}\bra{\frac{1}{s}-\frac{1}{s'}}} = \bra{-1}^{\al-1}\frac{1}{\prod\limits_{s' \neq s}\bra{\frac{1}{s'}-\frac{1}{s}}},\]
    and hence
    \[C_S= \bra{\prod\limits_{s \in S}\frac{1}{s}} \bra{-1}^{\al-1}\sum_{s \in S} c_{\frac{1}{S}}\bra{\frac{1}{s}}\bra{\frac{1}{s}}^{\al-1} =\bra{-1}^{\al-1} \bra{\prod\limits_{s \in S}\frac{1}{s}} ,\]
    since
    \[\sum_{s \in S} c_{\frac{1}{S}}\bra{\frac{1}{s}}\bra{\frac{1}{s}}^{\al-1} = 1\]
    by definition of the coefficients $c_{\frac{1}{S}}$.
    \bigbreak
Define
    \[F\bra{x}=\sum_{s \in S}k_S\bra{s}\bra{x+s}^{d+\al-1}-C_Sx^{d+\al-1}.\]
Expanding the $\bra{x+s}^{d+\al-1}$ term, the coefficient of $x^{d+\al-1}$ in
\[\sum_{s \in S}k_S\bra{s}\bra{x+s}^{d+\al-1}\]
is $C_S$, which cancels identically with the $C_Sx^{d+\al-1}$ term. So, $F(x)$ has no $x^{d+\al-1}$ term.
\smallbreak
For $1 \leq i \leq \al + d -1$, the $x^{\al+d-1-i}$ coefficient in $F(x)$ is
\[\binom{d+\al-1}{i}\sum_{s \in S}k_S\bra{s}s^{i}=\binom{d+\al-1}{i}\sum_{s \in S}c_S\bra{s}s^{i-1},\]
which is $0$ for $0 \leq i-1 \leq \al -2$, i.e. $1 \leq i \leq \al-1$. So, the first $\al$ coefficients of $F\bra{x}$ are $0$, and hence $F\bra{x}$ has degree at most $d-1$.
\smallbreak
We now show that $0$ is a root of $F$, of multiplicity at least $\al-1$. For $0 \leq j \leq \al-2$, the coefficient of $x^j$ in $F(x)$ is
\[\binom{d+\al-1}{j}\sum_{s \in S}k_S\bra{s}s^{d+\al-1-j}=\binom{d+\al-1}{j}\sum_{s \in S}c_S\bra{s}s^{\al-2-j}=0\]
by definition of the coefficients $c_S$, where we are using the fact that $s^d=1$. Thus every term in $F(x)$ has degree at least $\al-1$, so $x^{\al-1}$ divides $F(x)$, and hence $0$ is a root of $F$ of multiplicity at least $\al-1$.
\smallbreak
We now show that every $t \in T$ is also a root of $F$ of multiplicity at least $\al$. Similar to the previous part, we do this by showing that the first $\al$ coefficients of $y$ in the expansion of $F\bra{t+y}$ vanish. For $0 \leq j \leq \al-1$, the coefficient of $y^j$ in $F\bra{t+y}$ is
\[\binom{d+\al-1}{j}\bra{\sum_{s \in S}k_S\bra{s}\bra{s+t}^{d+\al-1-j}-C_S t^{d+\al-1-j}}.\]
Since $S+t \sub t\mu_d$, we have $\bra{s+t}^d=t^d$ for all $s \in S$, and hence the above is
\[t^d\binom{d+\al-1}{j}\bra{\sum_{s \in S}k_S\bra{s}\bra{s+t}^{\al-1-j}-C_S t^{\al-1-j}}.\]
Expanding $\bra{s+t}^{\al-1-j}$, we see that
\[\sum_{s \in S}k_S\bra{s}\bra{s+t}^{\al-1-j}-C_S t^{\al-1-j} = 0,\]
since all terms with a non-zero power of $s$ in the expansion of $\bra{s+t}^{\al-1-j}$ vanish by definition of the coefficients, and the $t^{\al-1-j}$ term cancels identically with the $C_S t^{\al-1-j}$ term. Thus every $t \in T$ is a root of $F$ with multiplicity at least $\al$.
\bigbreak
It follows that
\[F(x)= C_F x^{\al-1}\prod_{t \in T}\bra{x-t}^{\al},\]
since $x^{\al-1}\prod_{t \in T}\bra{x-t}^{\al}$ is certainly a factor of $F(x)$ by the above, but this has degree $\al-1+\ga\al = d-1$, and $F$ has degree at most $d-1$, so $F$ is in fact a scalar multiple of this factor.
\smallbreak
To find $C_F$, we compare coefficients of $x^{d-1}$. The coefficient of $x^{d-1}$ in the definition of $F$ is given by
\[\binom{d+\al-1}{\al}\sum_{s \in S}k_S(s)s^{\al}=\binom{d+\al-1}{\al}\sum_{s \in S}c_S(s)s^{\al-1}=\binom{d+\al-1}{\al}.\]
On the other hand, the coefficient of $x^{d-1}$ in $x^{\al-1}\prod_{t \in T}\bra{x-t}^{\al}$ is $1$, and hence
\[C_F=\binom{\al+d-1}{\al},\]
which proves the result.
\end{proof}

We now apply this identity to reciprocal translates of $A$. We fix $b \in B$, and from now on we will write
\[\frac{1}{A+b}=\left\{\frac{1}{a+b}:a \in A\right\}, \qquad B_b^{-1}= \left\{\frac{1}{b-b'}:b' \in B\setminus \{b\}\right\}.\]
For the remainder of the paper, suppose that $A+B=\mu_d$ is a proper multiplicative subgroup and that $\abs{A},\abs{B}>2$. By \thref{kalmyninsize}, we have $\abs{A}=\abs{B}=\al$ and $\al^2=d$.

\begin{lemma}\thlabel{poly2}
We have the following polynomial identity
\[\sum\limits_{a \in A} c_{\frac{1}{A+b}}\bra{\frac{1}{a+b}}\bra{a+b}\bra{x+\frac{1}{a+b}}^{d+\al-1}=C_{A+b}x^{d+\al-1}+\binom{d+\al-1}{\al}x^{\al-1}\prod\limits_{b' \neq b}\bra{x+\frac{1}{b-b'}}^{\al},\]
where
\[C_{A+b}=\bra{-1}^{\al-1}\prod_{a \in A}\bra{a+b}.\]
\end{lemma}
 \begin{remark} \thlabel{rem:yip} The above lemma is Kalmynin's \cite[Lemma 9]{kalmynin2025additive}. Kalmynin proves it by applying a universal change of variables $x\mapsto \frac1x+b$ to the original Hanson-Petridis polynomial. The same change of variables can also be used to prove the key claim by Kim, Yip and Yoo \cite[Claim 4.1]{kim2026multiplicative}. We instead derive it via \thref{modifiedHP} to emphasise the role of the relation \eqref{eq:alg} below. \end{remark}

\begin{proof}
Since $A+B=\mu_d$, we have $\frac{1}{A+b}\subseteq\mu_d$. Moreover, if $b'\neq b$, then
\begin{equation}\label{eq:alg}\frac{1}{a+b}+\frac{1}{b'-b}=\frac{1}{b'-b}\cdot\frac{a+b'}{a+b}\in\frac{1}{b'-b}\mu_d,\end{equation}
because both $a+b'$ and $a+b$ lie in $\mu_d$. Thus the hypotheses of \thref{modifiedHP} hold with $S=\frac{1}{A+b}$ and $T=-B_{b}^{-1}$, since $\abs{B_{b}^{-1}}=\alpha-1$ and $d=\alpha^2=\alpha\bra{\abs{B_{b}^{-1}}+1}$.
\smallbreak
Applying \thref{modifiedHP} with $S=\frac{1}{A+b}$ and $T=-B_b^{-1}$, we obtain
\[\sum\limits_{a\in A}k_{\frac{1}{A+b}}\bra{\frac{1}{a+b}}\bra{x+\frac{1}{a+b}}^{d+\alpha-1}-C_{\frac{1}{A+b}}x^{d+\alpha-1}=\binom{d+\alpha-1}{\alpha}x^{\alpha-1}\prod\limits_{b'\ne b}\bra{x-\frac{1}{b'-b}}^\alpha.\]
By definition of $k_{\frac{1}{A+b}}$, we have
\[k_{\frac{1}{A+b}}\bra{\frac{1}{a+b}}=c_{\frac{1}{A+b}}\bra{\frac{1}{a+b}}\bra{a+b}.\]
Also, by the formula for $C_S$ from \thref{modifiedHP},
\[C_{\frac{1}{A+b}}=(-1)^{\alpha-1}\prod\limits_{s\in \frac{1}{A+b}}\frac{1}{s}=(-1)^{\alpha-1}\prod\limits_{a\in A}\bra{a+b}=C_{A+b}.\]
Finally,
\[x-\frac{1}{b'-b}=x+\frac{1}{b-b'}.\]
Substituting these identities into the previous displayed equation and moving the term $C_{\frac{1}{A+b}}x^{d+\alpha-1}$ to the right-hand side gives the claimed identity.
\end{proof}

\begin{remark}
 Note that $\binom{\al+d-1}{\al}x^{\al-1}\prod\limits_{b' \neq b}\bra{x+\frac{1}{b-b'}}^\al$
is a polynomial of degree 
\[\al-1 + \al\bra{\be-1} = \al-1+\al\be-\al = d-1.\]
So, after $x^{\al+d-1}$, the next highest powers of $x$ in \thref{poly2} are $x^{d-1}, x^{d-2}, x^{d-3}$. We equated the $x^{d-1}$ coefficients in the proof of \thref{modifiedHP}, and will study the $x^{d-2}$ and $x^{d-3}$ coefficients in \thref{lineartransfer,quadtransfer}. 
\end{remark}

\begin{lemma}\thlabel{lineartransfer}
For any $b \in B$,
    \[p_1\bra{\Ab}=C_1 p_1\bra{\Bb},\]
    where 
    \[C_1 = \frac{\al}{\al-1}.\]
    \end{lemma}
\begin{proof}

    We equate the coefficients of $x^{d-2}$ in \thref{poly2}. The coefficient of $x^{d-2}$ on the left-hand side of \thref{poly2} is
    \[\binom{\al+d-1}{\al+1}\sum\limits_{a \in A}c_{\frac{1}{A+b}}\bra{\frac{1}{a+b}}\bra{\frac{1}{a+b}}^{\al}\]
    \[=\binom{\al+d-1}{\al+1}h_1\bra{\frac{1}{A+b}}\]
    by \thref{homo}.
    The coefficient of $x^{d-2}$ on the right-hand side of \thref{poly2} is
    \[\binom{\al+d-1}{\al} \al\sum\limits_{b' \neq b}\frac{1}{b-b'},\]
    and hence we have
    \[\binom{\al+d-1}{\al}\al\sum\limits_{b' \neq b}\frac{1}{b-b'} = \binom{\al+d-1}{\al+1}h_1\bra{\frac{1}{A+b}},\]
    which we can write as
    \[\sum\limits_{a \in A}\frac{1}{a+b}= \frac{\binom{\al+d-1}{\al}}{\binom{\al+d-1}{\al+1}}\al\sum\limits_{b' \neq b}\frac{1}{b-b'}\]
    \[=\frac{\al\bra{\al+1}}{d-1}\sum\limits_{b' \neq b}\frac{1}{b-b'}.\]
    Using the fact that $d = \al^2$ to simplify the quotient gives the result.
\end{proof}

\begin{lemma}\thlabel{quadtransfer}
For any $b \in B$,
\[p_1\bra{\Ab}^2 + p_2\bra{\Ab} = \frac{\al\bra{\al+2}}{\bra{\al-1}\bra{\al^2-2}}\bra{\al p_1\bra{\Bb}^2 - p_2\bra{\Bb}}.\]
\end{lemma}

\begin{proof}
We equate the coefficients of $x^{d-3}$ in the polynomial identity from \thref{poly2}.

By binomial expansion, the term containing $x^{d-3}$ on the left-hand side of \thref{poly2} is derived from the $x^{\alpha+d-1- (\alpha+2)} = x^{d-3}$ term inside the sum
\[ \sum_{a\in A}c_{\frac{1}{A+b}}\left(\frac{1}{a+b}\right)(a+b)\binom{d+\alpha-1}{\alpha+2}x^{d-3}\left(\frac{1}{a+b}\right)^{\alpha+2}.\]
The coefficient is therefore
\[ \binom{\alpha+d-1}{\alpha+2}\sum_{a\in A}c_{\frac{1}{A+b}}\left(\frac{1}{a+b}\right)(a+b)\left(\frac{1}{a+b}\right)^{\alpha+2} = \binom{\alpha+d-1}{\alpha+2}\sum_{a\in A}c_{\frac{1}{A+b}}\left(\frac{1}{a+b}\right)\left(\frac{1}{a+b}\right)^{\alpha+1}.\]
By \thref{homo}, this simplifies to
\[ \binom{\alpha+d-1}{\alpha+2}h_{2}\left(\frac{1}{A+b}\right).\]
Using the identity for $h_2$ from \thref{newton}, we get
\begin{equation}\label{3.4LHS}
\binom{\alpha+d-1}{\alpha+2}\frac{1}{2}\left(\left(\sum_{a\in A}\frac{1}{a+b}\right)^{2}+\sum_{a\in A}\frac{1}{(a+b)^{2}}\right) = \binom{\alpha+d-1}{\alpha+2}\frac{1}{2}\left(p_{1}\left(\frac{1}{A+b}\right)^{2}+p_{2}\left(\frac{1}{A+b}\right)\right).
\end{equation}

Next, we determine the coefficient on the right-hand side of \thref{poly2} - the relevant term is 
\[\binom{\al+d-1}{\al}x^{\alpha-1}\prod_{b^{\prime}\ne b}(x+\frac{1}{b-b^{\prime}})^{\alpha}.\]  
The expansion of the product yields 
\[\binom{\al+d-1}{\al} \left(\alpha^2 e_2(B^{-1}_b)+\frac{\alpha(\alpha-1)}{2}p_2(B^{-1}_b)\right) = \frac{\binom{\al+d-1}{\al}}{2}
\left(\alpha^2 p_1(B^{-1}_b)^2 -\alpha p_2(B^{-1}_b)\right).\]

Equating with the previous formula we have
\[ \binom{\alpha+d-1}{\alpha+2}\frac{1}{2}\left(p_{1}\left(\frac{1}{A+b}\right)^{2}+p_{2}\left(\frac{1}{A+b}\right)\right) = \frac{\binom{\alpha+d-1}{\alpha}\al}{2} \left(\alpha p_{1}(B_{b}^{-1})^{2}-p_{2}(B_{b}^{-1}) \right).\]
Simplifying the binomial coefficients gives the required result
\[ p_{1}\left(\frac{1}{A+b}\right)^{2}+p_{2}\left(\frac{1}{A+b}\right)=\frac{\al\bra{\alpha+2}}{\bra{\al-1}\bra{\al^2-2}}\left(\alpha p_{1}(B_{b}^{-1})^{2}-p_{2}(B_{b}^{-1})\right).\]
\end{proof}

\begin{corollary}\thlabel{identitiesII}
For any $b\in B$, we have
\[e_2\bra{B_b^{-1}}=C_{2,1}p_1\bra{\frac{1}{A+b}}^2+C_{2,2}p_2\bra{\frac{1}{A+b}},\]
where
\[C_{2,1}=\frac{\bra{\al-1}\bra{\al-2}}{2\al^2\bra{\al+2}},\qquad C_{2,2}=\frac{\bra{\al-1}\bra{\al^2-2}}{2\al\bra{\al+2}}.\]
Equivalently,
\[p_2\bra{B_b^{-1}}=C'_{2,1}p_1\bra{\frac{1}{A+b}}^2-C'_{2,2}p_2\bra{\frac{1}{A+b}},\]
where
\[C'_{2,1}=\frac{\al-1}{\al+2},\qquad C'_{2,2}=\frac{\bra{\al-1}\bra{\al^2-2}}{\al\bra{\al+2}}.\]
\end{corollary}
\begin{proof}
By \thref{lineartransfer}, we have
\[p_1\bra{B_b^{-1}}=\frac{\al-1}{\al}p_1\bra{\frac{1}{A+b}}.\]
Rearranging \thref{quadtransfer} gives
\[p_2\bra{B_b^{-1}}=\al p_1\bra{B_b^{-1}}^2-\frac{\bra{\al-1}\bra{\al^2-2}}{\al\bra{\al+2}}\bra{p_1\bra{\frac{1}{A+b}}^2+p_2\bra{\frac{1}{A+b}}}.\]
Substituting the formula for $p_1\bra{B_b^{-1}}$, we obtain
\[p_2\bra{B_b^{-1}}=\frac{\al-1}{\al+2}p_1\bra{\frac{1}{A+b}}^2-\frac{\bra{\al-1}\bra{\al^2-2}}{\al\bra{\al+2}}p_2\bra{\frac{1}{A+b}},\]
which proves the second identity.
Finally, since
\[e_2\bra{B_b^{-1}}=\frac12\bra{p_1\bra{B_b^{-1}}^2-p_2\bra{B_b^{-1}}},\]
we again use \thref{lineartransfer} and the formula just proved for $p_2\bra{B_b^{-1}}$. This gives
\[e_2\bra{B_b^{-1}}=\frac{\bra{\al-1}\bra{\al-2}}{2\al^2\bra{\al+2}}p_1\bra{\frac{1}{A+b}}^2+\frac{\bra{\al-1}\bra{\al^2-2}}{2\al\bra{\al+2}}p_2\bra{\frac{1}{A+b}},\]
as required.
\end{proof}

\section{The least nonzero power sum}\label{section4}

In this section, we introduce the least index at which a power sum of $A$ is nonzero and prove that this index is even. We organise the identities according to the number of reciprocal factors which they involve. The relations obtained directly by expanding powers of $a+b$ will be called \emph{zeroth-order} relations, while the argument using the first reciprocal transfer identity from \thref{lineartransfer} will be called \emph{first-order}.
\smallbreak
The parameter $k$ and the corresponding conclusions also appear in Kalmynin's treatment of quadratic residues. We include elementary proofs adapted to our notation and formulate the symmetrisation identity in the form which will be used again later.

\begin{lemma}\thlabel{klessthanalpha}
There exists an integer $1 \leq k \leq \al$ such that $p_k\bra{A}\neq0$.
\end{lemma}

\begin{proof}
Suppose for a contradiction that $p_i\bra{A}=0$ for every $1\leq i\leq\al$. By Newton's identities, this implies $e_i\bra{A}=0$ for every $1\leq i\leq\al$. Hence for an indeterminate $X$,
\[\prod_{a\in A}\bra{X-a}=X^\al.\]
This is impossible since $A$ is a set of size $\al>1$. Therefore there is some $1 \leq k\leq\al$ such that $p_k\bra{A} \neq 0$, as required.
\end{proof}
If $p_1\bra{A} \neq 0$, we can replace $A$ and $B$ by $A-t$ and $B+t$, where $t = \frac{p_1\bra{A}}{\al}$. This does not change $A+B$, so we can assume without loss of generality from now on that $p_1\bra{A}=0$.
For the rest of the paper, $k$ will denote the least positive integer such that $p_k\bra{A} \neq 0$, and by \thref{klessthanalpha} we have
\[2 \leq k \leq \al.\]

\subsection{The zeroth-order relation} 
The uniqueness of the representations in $A+B=\mu_d$ immediately relates the power sums of the two summands. 

\begin{lemma}\thlabel{zerothorder} 
For every $1\leq j\leq k$, 
\[p_j\bra{B}=-p_j\bra{A}.\] 
In particular, $k$ is also the least positive integer such that $p_k\bra{B}\neq0$, and \[p_k\bra{B}=-p_k\bra{A}\neq0.\] \end{lemma} 
\begin{proof} 
Fix $1\leq j\leq k$. Since $j\leq\alpha<d$, we have 
\[\sum_{w\in\mu_d}w^j=0.\] 
Each element of $\mu_d$ has a unique representation as $a+b$, and therefore 
\[0 =\sum_{a\in A,b\in B}\bra{a+b}^j\]
\[=\sum_{i=0}^j\binom{j}{i}p_i\bra{A}p_{j-i}\bra{B}. \] 
For $1\leq i\leq j-1$, we have $i<k$, and hence $p_i\bra{A}=0$. Thus only the two endpoint terms remain, giving 
\[0=\alpha p_j\bra{A}+\alpha p_j\bra{B}.\] 
Since $\alpha\neq0$ in $\F_p$, the result follows. 
\end{proof}

\subsection{First-order relations} \label{sec:fo}

We now introduce the reciprocal-difference sums used throughout the rest of the proof. For a finite set $B\sub\F_p$, an integer $t\geq0$ with $\abs{B}\geq t+1$, and an integer $k\geq-t$, define
\[S_{k,t}\bra{B}:=\sum^*_{B^{t+1}}
\frac{b_0^{k+t}}{\prod_{i=1}^t\bra{b_0-b_i}},\]
where $\sum^*_{B^{t+1}}$ denotes summation over all ordered
$\bra{t+1}$-tuples $\bra{b_0,\ldots,b_t}\in B^{t+1}$ {\em with pairwise distinct entries.} We use the convention that an empty
product is equal to $1$, and also set, for a set $C$, $h_k\bra{C}:=0$ when $k<0$.
\smallbreak

We have the following general formula for $S_{k,t}$, which we prove in the appendix.

\begin{lemma}\thlabel{generalsums}
Let $B\sub\F_p$ be a finite set, let $t\geq0$ with
$\abs{B}\geq t+1$, and let $k\geq-t$. Then
\[(t+1)S_{k,t}\bra{B}=\sum_{\substack{i_0+\cdots+i_t=k\\i_0,\ldots,i_t\geq0}}\sum^*_{B^{t+1}}b_0^{i_0}\cdots b_t^{i_t}.\]
Equivalently,
\[(t+1)S_{k,t}\bra{B}=(t+1)!\sum_{\substack{C\sub B\\\abs{C}=t+1}}h_k\bra{C}.\]
\end{lemma}

The following symmetrisation identity evaluates the first-order sums
in terms of power sums.

\begin{lemma}\thlabel{firstordersymmetrisation}
Let $T\sub\F_p$ be finite and let $m\geq0$. Then
\[2S_{m,1}\bra{T}=\sum_{i=0}^m p_{m-i}\bra{T}p_i\bra{T}-\bra{m+1}p_m\bra{T}.\]
\end{lemma}

\begin{proof}
Applying \thref{generalsums} with $t=1$ gives
\[2S_{m,1}\bra{T}=\sum_{i=0}^m\sum_{x,y\in T}^{*}x^{m-i}y^i.\]
For each $0\leq i\leq m$, we have
\[\sum_{x,y\in T}^{*}x^{m-i}y^i=p_{m-i}\bra{T}p_i\bra{T}-p_m\bra{T},\]
since the unrestricted sum is $p_{m-i}\bra{T}p_i\bra{T}$ and the diagonal contribution is
\[\sum_{x\in T}x^{m-i}x^i=p_m\bra{T}.\]
Summing over $i$ therefore gives
\[2S_{m,1}\bra{T}=\sum_{i=0}^m p_{m-i}\bra{T}p_i\bra{T}-\bra{m+1}p_m\bra{T},\]
as required.
\end{proof}

\begin{corollary}\thlabel{Sk1firstindex}
For $T=A$ and $T=B$, we have
\[S_{k,1}\bra{T}=\bra{\alpha-\frac{k+1}{2}}p_k\bra{T}.\]
\end{corollary}

\begin{proof}
By \thref{zerothorder}, $k$ is the least positive index for which $p_k\bra{T}\neq0$. Hence every term
\[p_{k-i}\bra{T}p_i\bra{T},\qquad 1\leq i\leq k-1,\]
vanishes. The two endpoint terms in \thref{firstordersymmetrisation} each equal $\alpha p_k\bra{T}$, and the result follows.
\end{proof}

We refer to the following as a \emph{first-order} relation, due to the use of \thref{lineartransfer}.

\begin{lemma}\thlabel{keven}
The integer $k$ is even.
\end{lemma}
\begin{proof}
Assume for a contradiction that $k$ is odd. By \thref{lineartransfer}, for any $b\in B$ we have
\[\sum_{a\in A}\frac{1}{a+b}=\frac{\al}{\al-1}\sum_{b'\in B\setminus\{b\}}\frac{1}{b-b'},\]
and hence
\begin{equation}\label{bkplusone}
\sum_{a\in A,b\in B}\frac{b^{k+1}}{a+b}=\frac{\al}{\al-1}S_{k,1}\bra{B}.
\end{equation}
Similarly,
\begin{equation}\label{akplusone}
\sum_{a\in A,b\in B}\frac{a^{k+1}}{a+b}=\frac{\al}{\al-1}S_{k,1}\bra{A}.
\end{equation}
Subtracting \eqref{bkplusone} from \eqref{akplusone}, we get
\[\sum_{a\in A,b\in B}\frac{a^{k+1}-b^{k+1}}{a+b}=\frac{\al}{\al-1}\bra{S_{k,1}\bra{A}-S_{k,1}\bra{B}}.\]
Since $k$ is odd, we have
\[\frac{a^{k+1}-b^{k+1}}{a+b}=a^k-a^{k-1}b+\cdots-b^k.\]
As $k$ is the first index for both $A$ and $B$, all mixed terms vanish after summing over $a\in A$ and $b\in B$. Therefore the left-hand side is
\[\al p_k\bra{A}-\al p_k\bra{B}=2\al p_k\bra{A},\]
using $p_k\bra{B}=-p_k\bra{A}$. On the other hand, by \thref{Sk1firstindex} and its analogue with $A$ and $B$ swapped,
\[S_{k,1}\bra{A}-S_{k,1}\bra{B}=\bra{\al-\frac{k+1}{2}}\bra{p_k\bra{A}-p_k\bra{B}}=2\bra{\al-\frac{k+1}{2}}p_k\bra{A}.\]
Hence
\[2\al p_k\bra{A}=\frac{2\al}{\al-1}\bra{\al-\frac{k+1}{2}}p_k\bra{A}.\]
Since $2\alpha p_k\bra{A}\neq0$ in $\F_p$, we may cancel it and obtain
\[1=\frac{\alpha-\frac{k+1}{2}}{\alpha-1},\]
and hence
\[k\equiv1\pmod p.\]
But $1\leq k\leq\alpha<p$, and hence $k=1$, which is a contradiction. Therefore $k$ is even.
\end{proof}

\section{Second-order relations}\label{section5}
Throughout this section, we retain the standing assumptions and notation from the previous sections. In particular, $\abs{A}=\abs{B}=\al$, and $k$ denotes the least positive integer such that $p_k\bra{A}\neq0$. We have $2\leq k\leq\al$, and by \thref{keven}, the integer $k$ is even. Our goal in this section is to prove the following polynomial congruence.

\begin{prop}\thlabel{firstpoly}
We have
\[P\bra{\al,k}\equiv0\pmod p,\]
where
\[P\bra{\al,x}=3\bra{x^2-3x-2}\al^2+4\bra{x^2+2}\al-\bra{x^2-3x+2}.\]
\end{prop}

We define
\begin{align*}
X_m\bra{B}&=\sum_{\substack{b\in B\\ a,a'\in A\\ a\neq a'}}\frac{b^{m+2}}{\bra{b+a}\bra{b+a'}},\\
Y_m\bra{B}&=\sum_{a\in A,b\in B}\frac{b^{m+2}}{\bra{b+a}^2}.
\end{align*}
We define $X_m\bra{A}$ and $Y_m\bra{A}$ analogously by swapping $A$ and $B$.
\smallbreak
We refer to the relations in this section as \emph{second-order} relations, since we use \thref{quadtransfer}, in addition to \thref{lineartransfer}.

\begin{lemma}\thlabel{secondorderequation}
For every $m\geq 1$, we have
\[S_{m,2}\bra{B}=\frac{\bra{\al-1}\bra{\al-2}}{\al^2\bra{\al+2}}X_m\bra{B}+\frac{\bra{\al-1}\bra{\al^3-\al-2}}{\al^2\bra{\al+2}}Y_m\bra{B}.\]
\end{lemma}
\begin{proof}
By definition,
\[S_{m,2}\bra{B}=2\sum_{b\in B}b^{m+2}e_2\bra{B_b^{-1}}.\]
Using \thref{identitiesII}, we have
\[e_2\bra{B_b^{-1}}=C_{2,1}p_1\bra{\frac{1}{A+b}}^2+C_{2,2}p_2\bra{\frac{1}{A+b}}.\]
Therefore
\[S_{m,2}\bra{B}=2C_{2,1}\sum_{b\in B}b^{m+2}p_1\bra{\frac{1}{A+b}}^2+2C_{2,2}\sum_{b\in B}b^{m+2}p_2\bra{\frac{1}{A+b}}.\]
Now
\[\sum_{b\in B}b^{m+2}p_1\bra{\frac{1}{A+b}}^2=X_m\bra{B}+Y_m\bra{B},\]
and
\[\sum_{b\in B}b^{m+2}p_2\bra{\frac{1}{A+b}}=Y_m\bra{B}.\]
Hence
\[S_{m,2}\bra{B}=2C_{2,1}X_m\bra{B}+2\bra{C_{2,1}+C_{2,2}}Y_m\bra{B}.\]
Substituting
\[C_{2,1}=\frac{\bra{\al-1}\bra{\al-2}}{2\al^2\bra{\al+2}},\qquad C_{2,2}=\frac{\bra{\al-1}\bra{\al^2-2}}{2\al\bra{\al+2}},\]
we obtain
\[2C_{2,1}=\frac{\bra{\al-1}\bra{\al-2}}{\al^2\bra{\al+2}},\qquad 2\bra{C_{2,1}+C_{2,2}}=\frac{\bra{\al-1}\bra{\al^3-\al-2}}{\al^2\bra{\al+2}},\]
which gives the claimed identity.
\end{proof}

We will use the combinatorial identity \thref{generalsums}, proved in the appendix, in the proof of the following.
\begin{lemma}\thlabel{Sk2firstindex}
We have
\[S_{k,2}\bra{B}=\bra{\al^2-\bra{k+2}\al+\frac{\bra{k+1}\bra{k+2}}{3}}p_k\bra{B}.\]
\end{lemma}
\begin{proof} Applying \thref{generalsums} with $t=2$ gives 
\[3S_{k,2}\bra{B}=\sum_{\substack{i+j+\ell=k\\i,j,\ell\geq0}}\sum_{x,y,z\in B}^{*}x^iy^jz^\ell.\] 
For fixed $i,j,\ell$, inclusion-exclusion over the equalities $x=y$, $x=z$ and $y=z$ gives 
\[\sum_{x,y,z\in B}^{*}x^iy^jz^\ell=p_i\bra{B}p_j\bra{B}p_\ell\bra{B}-p_{i+j}\bra{B}p_\ell\bra{B}-p_{i+\ell}\bra{B}p_j\bra{B}-p_i\bra{B}p_{j+\ell}\bra{B}+2p_k\bra{B}.\] 
Since $k$ is the least positive integer for which $p_k\bra{B}\neq0$, the only nonzero terms in 
\[\sum_{\substack{i+j+\ell=k\\i,j,\ell\geq0}}p_i\bra{B}p_j\bra{B}p_\ell\bra{B}\] 
are the three permutations of $\bra{k,0,0}$. As $p_0\bra{B}=\alpha$, their total contribution is 
\[3\alpha^2p_k\bra{B}.\] Moreover, 
\[\sum_{\substack{i+j+\ell=k\\i,j,\ell\geq0}}p_{i+j}\bra{B}p_\ell\bra{B}=\sum_{u=0}^k\bra{u+1}p_u\bra{B}p_{k-u}\bra{B}=\bra{k+2}\alpha p_k\bra{B}.\] 
Indeed, the terms with $1\leq u\leq k-1$ vanish, while the terms $u=0$ and $u=k$ contribute $\alpha p_k\bra{B}$ and $\bra{k+1}\alpha p_k\bra{B}$, respectively. The other two pairwise-equality sums have the same value. Finally, there are $\binom{k+2}{2}$ triples of nonnegative integers satisfying $i+j+\ell=k$, and hence 
\[\sum_{\substack{i+j+\ell=k\\i,j,\ell\geq0}}2p_k\bra{B}=\bra{k+1}\bra{k+2}p_k\bra{B}.\] 
Combining these evaluations gives 
\[3S_{k,2}\bra{B}=\left(3\alpha^2-3\bra{k+2}\alpha+\bra{k+1}\bra{k+2}\right)p_k\bra{B}.\] 
Dividing by $3$ proves the result. 
\end{proof}

\subsection{Symmetric identities}
For brevity, we write
\[X_k^+:=X_k\bra{A}+X_k\bra{B},\qquad Y_k^+:=Y_k\bra{A}+Y_k\bra{B}.\]
Our goal in this section is to show the following. 
\begin{lemma}\thlabel{secondordersymvanish}
We have
\[X_k^+=Y_k^+=0.\]
\end{lemma}
We achieve this by obtaining two relations between $X_k^+$ and $Y_k^+$.

\begin{lemma}\thlabel{secondordersymone}
We have
\[(\al-2)X_k^++\bra{\al^3-\al-2}Y_k^+=0.\]
\end{lemma}

\begin{proof}
By \thref{secondorderequation} with $m=k$, we have
\[S_{k,2}\bra{B}=\frac{\bra{\al-1}\bra{\al-2}}{\al^2\bra{\al+2}}X_k\bra{B}+\frac{\bra{\al-1}\bra{\al^3-\al-2}}{\al^2\bra{\al+2}}Y_k\bra{B}.\]
Swapping $A$ and $B$ gives
\[S_{k,2}\bra{A}=\frac{\bra{\al-1}\bra{\al-2}}{\al^2\bra{\al+2}}X_k\bra{A}+\frac{\bra{\al-1}\bra{\al^3-\al-2}}{\al^2\bra{\al+2}}Y_k\bra{A}.\]
Adding these two equations, and using \thref{Sk2firstindex} together with $p_k\bra{A}=-p_k\bra{B}$, gives
\[0=\frac{\bra{\al-1}\bra{\al-2}}{\al^2\bra{\al+2}}X_k^++\frac{\bra{\al-1}\bra{\al^3-\al-2}}{\al^2\bra{\al+2}}Y_k^+.\]
Multiplying by $\al^2\bra{\al+2}/\bra{\al-1}$ gives
\[(\al-2)X_k^++\bra{\al^3-\al-2}Y_k^+=0,\]
as required.
\end{proof}

\begin{lemma}\thlabel{secondordersymtwo}
We have
\[(3\al-2)X_k^++2\bra{\al-1}Y_k^+=0.\]
\end{lemma}

\begin{proof}
We first prove that
\[-X_k\bra{B}=\bra{\al^2-k\al}p_k\bra{B}+\frac{2\bra{\al-1}}{\al}\bra{X_k\bra{A}+Y_k\bra{A}}.\]
By partial fractions,
\[\frac{1}{\bra{b+a}\bra{b+a'}}=\frac{1}{a-a'}\bra{\frac{1}{b+a'}-\frac{1}{b+a}}.\]
Therefore, after swapping $a$ and $a'$ in the first term,
\[X_k\bra{B}=-2\sum_{\substack{b\in B\\a,a'\in A\\a\neq a'}}\frac{b^{k+2}}{\bra{a-a'}\bra{a+b}}.\]
Since $k$ is even, we have
\[\frac{b^{k+2}-a^{k+2}}{a+b}=\sum_{i=0}^{k+1}\bra{-1}^ib^{k+1-i}a^i.\]
Thus
\[-X_k\bra{B}=2\sum_{\substack{b\in B\\a,a'\in A\\a\neq a'}}\sum_{i=0}^{k+1}\bra{-1}^i\frac{b^{k+1-i}a^i}{a-a'}+2\sum_{\substack{b\in B\\a,a'\in A\\a\neq a'}}\frac{a^{k+2}}{\bra{a-a'}\bra{a+b}}.\]
We evaluate the first term. The $i=0$ contribution vanishes by antisymmetry in $a,a'$. The $i=1$ contribution is
\[-p_k\bra{B}\sum_{\substack{a,a'\in A\\a\neq a'}}\frac{a}{a-a'}=-\frac{\al\bra{\al-1}}{2}p_k\bra{B}.\]
For $2\leq i\leq k$, the contribution vanishes since $k$ is the first index for $B$. The $i=k+1$ contribution is
\[-\al S_{k,1}\bra{A}.\]
Using \thref{Sk1firstindex} and $p_k\bra{A}=-p_k\bra{B}$, this is
\[-\al S_{k,1}\bra{A}=\al\bra{\al-\frac{k+1}{2}}p_k\bra{B}.\]
Hence the first term contributes
\[\bra{-\frac{\al\bra{\al-1}}{2}+\al\bra{\al-\frac{k+1}{2}}}2p_k\bra{B}=\bra{\al^2-k\al}p_k\bra{B}.\]
For the second term, by \thref{lineartransfer} with $A$ and $B$ swapped, we have
\[\sum_{a'\in A\setminus\{a\}}\frac{1}{a-a'}=\frac{\al-1}{\al}\sum_{b\in B}\frac{1}{a+b}.\]
Therefore
\[2\sum_{\substack{b\in B\\a,a'\in A\\a\neq a'}}\frac{a^{k+2}}{\bra{a-a'}\bra{a+b}}=\frac{2\bra{\al-1}}{\al}\sum_{a\in A}a^{k+2}\bra{\sum_{b\in B}\frac{1}{a+b}}^2.\]
By definition,
\[\sum_{a\in A}a^{k+2}\bra{\sum_{b\in B}\frac{1}{a+b}}^2=X_k\bra{A}+Y_k\bra{A}.\]
This proves
\[-X_k\bra{B}=\bra{\al^2-k\al}p_k\bra{B}+\frac{2\bra{\al-1}}{\al}\bra{X_k\bra{A}+Y_k\bra{A}}.\]
Swapping $A$ and $B$ gives
\[-X_k\bra{A}=\bra{\al^2-k\al}p_k\bra{A}+\frac{2\bra{\al-1}}{\al}\bra{X_k\bra{B}+Y_k\bra{B}}.\]
Adding these two identities and using $p_k\bra{A}=-p_k\bra{B}$ gives
\[-X_k^+=\frac{2\bra{\al-1}}{\al}\bra{X_k^++Y_k^+}.\]
Multiplying by $\al$ and rearranging gives
\[(3\al-2)X_k^++2\bra{\al-1}Y_k^+=0,\]
as required.
\end{proof}

We can now use \thref{secondordersymone,secondordersymtwo} to prove \thref{secondordersymvanish}.

\begin{proof}[Proof of \thref{secondordersymvanish}]
By \thref{secondordersymone} and \thref{secondordersymtwo}, we have
\[(\al-2)X_k^++\bra{\al^3-\al-2}Y_k^+=0,\]
and
\[(3\al-2)X_k^++2\bra{\al-1}Y_k^+=0.\]
The determinant of this system is
\[
\begin{vmatrix}
\al-2 & \al^3-\al-2\\
3\al-2 & 2\bra{\al-1}
\end{vmatrix}
=
2\bra{\al-2}\bra{\al-1}-\bra{3\al-2}\bra{\al^3-\al-2}.
\]
Expanding and simplifying,
\[2\bra{\al-2}\bra{\al-1}-\bra{3\al-2}\bra{\al^3-\al-2}=-\al\bra{3\al^3-2\al^2-5\al+2}.\]
Since
\[3\al^3-2\al^2-5\al+2=\bra{\al-1}\bra{3\al^2+\al-2},\]
the determinant is
\[-\al\bra{\al-1}\bra{3\al^2+\al-2}=-\al\bra{\al-1}\bra{3\al-2}\bra{\al+1}.\]
Under the standing hypotheses, we have $2<\al<p$ and $\al^2\leq p-1$. Thus $\al$, $\al-1$, and $\al+1$ are non-zero in $\F_p$. Also
\[0<3\al-2<\al^2+1\leq p,\]
since $\al\geq3$. Hence $3\al-2$ is also non-zero in $\F_p$. Therefore the determinant is non-zero, so the only solution is
\[X_k^+=Y_k^+=0,\]
as required.
\end{proof}

\subsection{Antisymmetric identities}
\begin{lemma}\thlabel{Yk2firstindex}
We have
\[Y_k\bra{B}=\frac{\al\bra{4\al-k^2-k-2}}{4\bra{\al-1}}p_k\bra{B}.\]
\end{lemma}
\begin{proof}
By definition,
\[Y_k\bra{B}-Y_k\bra{A}=\sum_{a\in A,b\in B}\frac{b^{k+2}-a^{k+2}}{\bra{a+b}^2}.\]
Since $k$ is even, we have
\[\frac{b^{k+2}-a^{k+2}}{a+b}=\sum_{i=0}^{k+1}\bra{-1}^ib^{k+1-i}a^i.\]
Therefore
\[Y_k\bra{B}-Y_k\bra{A}=\sum_{i=0}^{k+1}\bra{-1}^i\sum_{a\in A,b\in B}\frac{b^{k+1-i}a^i}{a+b}.\]
The term $i=0$ is
\[\sum_{a\in A,b\in B}\frac{b^{k+1}}{a+b}=\frac{\al}{\al-1}S_{k,1}\bra{B},\]
by \thref{lineartransfer}. For $1\leq i\leq k+1$, we use
\[\frac{b^{k+1-i}a^i}{a+b}=\sum_{j=0}^{i-1}\bra{-1}^ja^{i-1-j}b^{k+1-i+j}+\bra{-1}^i\frac{b^{k+1}}{a+b}.\]
Summing over $a\in A$ and $b\in B$, and using that $k$ is the first index for both $A$ and $B$, the polynomial part contributes only for $i=1$ and $i=k+1$. Thus
\[\sum_{i=1}^{k+1}\bra{-1}^i\sum_{j=0}^{i-1}\bra{-1}^jp_{i-1-j}\bra{A}p_{k+1-i+j}\bra{B}=-\al p_k\bra{B}+\al p_k\bra{A}.\]
Since $p_k\bra{A}=-p_k\bra{B}$, this contribution is $-2\al p_k\bra{B}$. The remaining reciprocal term contributes
\[\sum_{i=1}^{k+1}\bra{-1}^i\bra{-1}^i\sum_{a\in A,b\in B}\frac{b^{k+1}}{a+b}=\bra{k+1}\frac{\al}{\al-1}S_{k,1}\bra{B}.\]
Including the $i=0$ term, we obtain
\[Y_k\bra{B}-Y_k\bra{A}=\frac{\al\bra{k+2}}{\al-1}S_{k,1}\bra{B}-2\al p_k\bra{B}.\]
Using \thref{Sk1firstindex}, this becomes
\[Y_k\bra{B}-Y_k\bra{A}=\bra{\frac{\al\bra{k+2}}{\al-1}\bra{\al-\frac{k+1}{2}}-2\al}p_k\bra{B}.\]
Simplifying gives
\[Y_k\bra{B}-Y_k\bra{A}=\frac{\al\bra{4\al-k^2-k-2}}{2\bra{\al-1}}p_k\bra{B}.\]
By \thref{secondordersymvanish}, we have $Y_k\bra{A}=-Y_k\bra{B}$. Hence
\[Y_k\bra{B}=\frac{\al\bra{4\al-k^2-k-2}}{4\bra{\al-1}}p_k\bra{B},\]
as required.
\end{proof}

\begin{lemma}\thlabel{Xk2firstindex}
We have
\[X_k\bra{B}=\frac{\al\bra{2\al^2-2\bra{k+2}\al+k^2+k+2}}{2\bra{\al-2}}p_k\bra{B}.\]
\end{lemma}

\begin{proof}
From the proof of \thref{secondordersymtwo}, we have
\[-X_k\bra{B}=\bra{\al^2-k\al}p_k\bra{B}+\frac{2\bra{\al-1}}{\al}\bra{X_k\bra{A}+Y_k\bra{A}}.\]
Using \thref{secondordersymvanish}, this becomes
\[-X_k\bra{B}=\bra{\al^2-k\al}p_k\bra{B}-\frac{2\bra{\al-1}}{\al}\bra{X_k\bra{B}+Y_k\bra{B}}.\]
Rearranging gives
\[\frac{\al-2}{\al}X_k\bra{B}=\bra{\al^2-k\al}p_k\bra{B}-\frac{2\bra{\al-1}}{\al}Y_k\bra{B}.\]
Substituting \thref{Yk2firstindex}, we get
\[\frac{\al-2}{\al}X_k\bra{B}=\bra{\al^2-k\al}p_k\bra{B}-\frac{2\bra{\al-1}}{\al}\frac{\al\bra{4\al-k^2-k-2}}{4\bra{\al-1}}p_k\bra{B}.\]
Thus
\[\frac{\al-2}{\al}X_k\bra{B}=\frac{2\al^2-2k\al-4\al+k^2+k+2}{2}p_k\bra{B}.\]
Multiplying by $\al/\bra{\al-2}$, which is valid since $\alpha>2$, gives
\[X_k\bra{B}=\frac{\al\bra{2\al^2-2\bra{k+2}\al+k^2+k+2}}{2\bra{\al-2}}p_k\bra{B},\]
as required.
\end{proof}

We can now prove \thref{firstpoly}.
\begin{proof}[Proof of \thref{firstpoly}]
By \thref{secondorderequation} with $m=k$, we have
\[S_{k,2}\bra{B}=\frac{\bra{\al-1}\bra{\al-2}}{\al^2\bra{\al+2}}X_k\bra{B}+\frac{\bra{\al-1}\bra{\al^3-\al-2}}{\al^2\bra{\al+2}}Y_k\bra{B}.\]
Substituting \thref{Sk2firstindex}, \thref{Xk2firstindex} and \thref{Yk2firstindex}, and using $p_k\bra{B}\neq 0$, gives
\begin{multline}\label{firstpoly1}\al^2-\bra{k+2}\al+\frac{\bra{k+1}\bra{k+2}}{3}=\frac{\bra{\al-1}\bra{\al-2}}{\al^2\bra{\al+2}}\cdot\frac{\al\bra{2\al^2-2\bra{k+2}\al+k^2+k+2}}{2\bra{\al-2}}\\+\frac{\bra{\al-1}\bra{\al^3-\al-2}}{\al^2\bra{\al+2}}\cdot\frac{\al\bra{4\al-k^2-k-2}}{4\bra{\al-1}}.\end{multline}
The right-hand side of \eqref{firstpoly1} simplifies to
\[\frac{\bra{\al-1}\bra{2\al^2-2\bra{k+2}\al+k^2+k+2}}{2\al\bra{\al+2}}+\frac{\bra{\al^3-\al-2}\bra{4\al-k^2-k-2}}{4\al\bra{\al+2}}.\]
Multiplying by $12\al\bra{\al+2}$, and simplifying, gives
\[3\bra{k^2-3k-2}\al^2+4\bra{k^2+2}\al-\bra{k^2-3k+2}=0,\]
as required.
\end{proof}

\begin{remark}
    One can study the \emph{third-order} relations, by matching the $x^{d-4}$ coefficients in \thref{poly2}, and hence obtain a second polynomial relation in a similar manner to the approach in this section. Theoretically, one could in fact repeat this for every coefficient, obtaining an arbitrary number of polynomial relations. Unlike \thref{firstpoly}, the polynomial coming from the third-order relations has degree $12$ in $\al$, and degree $3$ in $k$. 
    \smallbreak
    With significant effort, we did compute this polynomial, which yields a lengthy proof of the forthcoming \thref{kdividesalpha}, but the utility of repeating this for the fourth-order relations seems highly questionable, as the resulting quartic polynomial in $k$ would most likely have a forbiddingly high degree in $\al$. Somewhat unexpectedly, the polynomial in \thref{firstpoly} is only quadratic in $\al$, which makes it possible to directly analyse it, as we do in Section \ref{section7}.
\end{remark}

\section{$k$ divides $\al$}\label{section6}
In this section, we prove that the power sums of $A$ and $B$ vanish at every index $i\leq\alpha$ which is not divisible by $k$. Newton's identities will then imply that $k\mid\alpha$. We retain the assumptions from the previous sections, in particular that $H \leq \F_p^*$ is a multiplicative subgroup such that $H = A + B$, where $\abs{A}=\abs{B} = \al >2$, and that $k \geq 2$ is the least positive integer such that $\sum_{a \in A}a^k \neq 0$.
\begin{prop}\thlabel{kdividesalpha}
We have $k\mid\al$.
\end{prop}

We first assume that there is some other integer $1 \leq m \leq \al$ such that $p_m\bra{A} \neq 0$ and $m$ is not a multiple of $k$.
\begin{lemma}\thlabel{leastnonmultiple}
Suppose that there exists an integer $1 \leq i \leq \al$ such that $k\nmid i$ and $p_i\bra{A}\neq0$. Let $m$ be the least such integer. Then the following hold.
\begin{enumerate}
    \item If $0<j<m$ and $k\nmid j$, then
    \[p_j\bra{A}=p_j\bra{B}=0.\]
    \item We have
    \[p_m\bra{A}=-p_m\bra{B}.\]
    \item If $i_1,\ldots,i_r$ are positive integers with $i_1+\cdots+i_r=m$ and $i_j<m$ for every $j$, then
    \[p_{i_1}\bra{T_1}\cdots p_{i_r}\bra{T_r}=0\]
    for every choice of $T_1,\ldots,T_r\in\{A,B\}$.
    \item The integer $m$ is even.
\end{enumerate}
\end{lemma}
\begin{proof}
The vanishing $p_j\bra{A}=0$ for $0<j<m$ with $k\nmid j$ follows immediately from the minimality of $m$. We prove the corresponding statement for $B$ by induction on $j$.
\smallbreak
Let $0<j<m $ with $k\nmid j$, and assume that $p_{j'}\bra{B}=0$ for every $0<j'<j$ with $k\nmid j'$. Expanding $\sum_{a\in A,b\in B}\bra{a+b}^j=0$, we get
\[0=\sum_{\ell=0}^{j}\binom{j}{\ell}p_\ell\bra{A}p_{j-\ell}\bra{B}.\]
Every mixed term vanishes. Indeed, if $0<\ell<j$ and $k\nmid\ell$, then $p_\ell\bra{A}=0$ by the minimality of $m$. If $k\mid\ell$, then $k\nmid j-\ell$, so $p_{j-\ell}\bra{B}=0$ by the induction hypothesis. Hence
\[0=\be p_j\bra{A}+\al p_j\bra{B}.\]
Since $p_j\bra{A}=0$, it follows that $p_j\bra{B}=0$. This proves the first claim.
\smallbreak
Now expand $\sum_{a\in A,b\in B}\bra{a+b}^m=0$. We obtain
\[0=\sum_{\ell=0}^{m}\binom{m}{\ell}p_\ell\bra{A}p_{m-\ell}\bra{B}.\]
Again all mixed terms vanish. Indeed, if $0<\ell<m$ and both $\ell$ and $m-\ell$ were multiples of $k$, then $m$ would be a multiple of $k$, contrary to the choice of $m$. Otherwise one of the two lower indices is a nonmultiple of $k$, so the corresponding power sum vanishes by the first claim. Therefore
\[0=\al p_m\bra{A}+\al p_m\bra{B}.\]
and hence $p_m\bra{A}=-p_m\bra{B}$.
\smallbreak
We next prove the lower-product vanishing. Suppose $i_1+\cdots+i_r=m$, with $0<i_j<m$ for every $j$, and suppose
\[p_{i_1}\bra{T_1}\cdots p_{i_r}\bra{T_r}\neq0\]
for some $T_1,\ldots,T_r\in\{A,B\}$. Then each factor $p_{i_j}\bra{T_j}$ is non-zero. By the first claim, each $i_j$ must be a multiple of $k$. Hence $m=i_1+\cdots+i_r$ is a multiple of $k$, contradiction. This proves the third claim.
\smallbreak
It remains to prove that $m$ is even. Suppose for a contradiction that $m$ is odd. By \thref{lineartransfer},
\[\sum_{a\in A,b\in B}\frac{b^{m+1}}{a+b}=\frac{\al}{\al-1}S_{m,1}\bra{B},\]
and, after interchanging $A$ and $B$,
\[\sum_{a\in A,b\in B}\frac{a^{m+1}}{a+b}=\frac{\al}{\al-1}S_{m,1}\bra{A}.\]
Subtracting these identities gives
\[\sum_{a\in A,b\in B}\frac{a^{m+1}-b^{m+1}}{a+b}=\frac{\al}{\al-1}\bra{S_{m,1}\bra{A}-S_{m,1}\bra{B}}.\]

Since $m$ is odd, we have
\[\frac{a^{m+1}-b^{m+1}}{a+b}=\sum_{i=0}^m\bra{-1}^i a^{m-i}b^i.\]
After summing over $a\in A$ and $b\in B$, every term with $1\leq i\leq m-1$ vanishes by the lower-product vanishing proved above. Hence the left-hand side is
\[\al p_m\bra{A}-\al p_m\bra{B}=2\al p_m\bra{A}.\]

On the other hand, \thref{firstordersymmetrisation} gives
\[2S_{m,1}\bra{T}=\sum_{i=0}^m p_{m-i}\bra{T}p_i\bra{T}-\bra{m+1}p_m\bra{T}\]
for $T=A,B$. Again, every term with $1\leq i\leq m-1$ vanishes by the lower-product vanishing. Since $p_0\bra{T}=\al$, we obtain
\[S_{m,1}\bra{T}=\bra{\al-\frac{m+1}{2}}p_m\bra{T}.\]
Using $p_m\bra{B}=-p_m\bra{A}$, it follows that
\[S_{m,1}\bra{A}-S_{m,1}\bra{B}=2\bra{\al-\frac{m+1}{2}}p_m\bra{A}.\]
Substituting into the preceding transfer identity gives
\[2\al p_m\bra{A}=\frac{2\al}{\al-1}\bra{\al-\frac{m+1}{2}}p_m\bra{A}.\]
Since $\al$, $\al-1$ and $p_m\bra{A}$ are nonzero in $\F_p$, we may cancel to obtain
\[1=\frac{\al-\frac{m+1}{2}}{\al-1}.\]
Therefore
\[m\equiv1\pmod p.\]
But $1\leq m\leq\al<p$, so $m=1$. This contradicts $p_1\bra{A}=0$. Hence $m$ is even.

\end{proof}

\begin{lemma}\thlabel{m}
Let $m \leq \al$ be the least positive integer such that $k\nmid m$ and $p_m\bra{A}\neq0$. Let $P\bra{\al,x}$ be the polynomial from \thref{firstpoly}. Then
\[P\bra{\al,m}\equiv0\pmod{p}.\]
\end{lemma}
\begin{proof} 
By \thref{leastnonmultiple}, we have 
\[p_m\bra{A}=-p_m\bra{B}\neq0,\] 
the integer $m$ is even, and every product involving at least two positive-index power sums whose indices have total degree $m$ vanishes. The identity in \thref{secondorderequation} holds for every exponent, and therefore 
\[S_{m,2}\bra{B}=\frac{\bra{\al-1}\bra{\al-2}}{\al^2\bra{\al+2}}X_m\bra{B}+\frac{\bra{\al-1}\bra{\al^3-\al-2}}{\al^2\bra{\al+2}}Y_m\bra{B}.\] 
Applying \thref{generalsums} with $t=2$ gives 
\[S_{m,2}\bra{T}=\bra{\al^2-\bra{m+2}\al+\frac{\bra{m+1}\bra{m+2}}{3}}p_m\bra{T}\] 
for $T=A,B$. Indeed, after inclusion-exclusion, every term other than those containing $p_0\bra{T}=\al$ and $p_m\bra{T}$ is a product of positive-index power sums whose indices sum to $m$, and hence vanishes by \thref{leastnonmultiple}. We may now repeat the symmetric calculations of Section \ref{section5}. Adding the two identities obtained from \thref{secondorderequation} gives 
\[\bra{\al-2}\bra{X_m\bra{A}+X_m\bra{B}}+\bra{\al^3-\al-2}\bra{Y_m\bra{A}+Y_m\bra{B}}=0.\] 
The proof of \thref{secondordersymtwo}, with $m$ in place of $k$, gives 
\[\bra{3\al-2}\bra{X_m\bra{A}+X_m\bra{B}}+2\bra{\al-1}\bra{Y_m\bra{A}+Y_m\bra{B}}=0.\] 
Here the only terms discarded in that proof are products involving at least two positive-index power sums whose indices sum to $m$, so they vanish by \thref{leastnonmultiple}. Since the determinant of this system is nonzero, as shown in the proof of \thref{secondordersymvanish}, we obtain 
\[X_m\bra{A}+X_m\bra{B}=Y_m\bra{A}+Y_m\bra{B}=0.\] 
The antisymmetric calculations in \thref{Yk2firstindex,Xk2firstindex} also use only the parity of the exponent, the relation between the two power sums, and the same lower-product vanishing. Replacing $k$ by $m$ therefore gives 
\[Y_m\bra{B}=\frac{\al\bra{4\al-m^2-m-2}}{4\bra{\al-1}}p_m\bra{B}\] and 
\[X_m\bra{B}=\frac{\al\bra{2\al^2-2\bra{m+2}\al+m^2+m+2}}{2\bra{\al-2}}p_m\bra{B}.\] 
Substituting these formulas and the preceding expression for $S_{m,2}\bra{B}$ into \thref{secondorderequation}, and cancelling $p_m\bra{B}\neq0$, gives 
\[3\bra{m^2-3m-2}\al^2+4\bra{m^2+2}\al-\bra{m^2-3m+2}=0.\] Thus 
\[P\bra{\al,m}\equiv0\pmod p,\] 
as required. 
\end{proof}

\begin{lemma}\thlabel{rootuniqueness}
Let $\al>2$, let $M\geq 2$, and suppose that $p=M\al^2+1$ is prime. Then the congruence
\[P\bra{\al,x}\equiv 0\pmod p\]
has at most one solution with $1\leq x\leq \al$.
\end{lemma}

\begin{proof}
Suppose for a contradiction that $r$ and $s$ are two distinct solutions with
\[1\leq r<s\leq\alpha,\]
and let $\sigma=r+s$. Then
\[3\leq\sigma\leq2\alpha-1.\]
Subtracting the two congruences gives
\[P\bra{\alpha,s}-P\bra{\alpha,r} =\bra{s^2-r^2}\bra{3\alpha^2+4\alpha-1}+\bra{s-r}\bra{-9\alpha^2+3}\]
\[=\bra{s-r}\left(\bra{3\alpha^2+4\alpha-1}\sigma-9\alpha^2+3\right).\]
Since $P\bra{\alpha,r}\equiv P\bra{\alpha,s}\equiv0\pmod p$ and $0<s-r<\alpha<p$, we may cancel $s-r$ modulo $p$. Hence
\[p\mid\bra{3\sigma-9}\alpha^2+4\sigma\alpha+3-\sigma.\]
The integer on the right is positive, so there is a positive integer $\nu$ such that
\[\bra{3\sigma-9}\alpha^2+4\sigma\alpha+3-\sigma=\nu\bra{M\alpha^2+1}.\]
Let
\[\eta=\nu M-\bra{3\sigma-9}.\]
Rearranging the preceding identity gives
\[\eta\alpha^2=4\sigma\alpha+3-\sigma-\nu.\]

We first show that
\[1\leq\eta\leq7.\]
Since $M\alpha^2+1>\alpha^2$, we have
\[\nu<3\sigma-9+\frac{4\sigma}{\alpha}+\frac{3-\sigma}{\alpha^2}.\]
Using $\sigma\leq2\alpha-1$, we obtain
\[\nu<3\sigma-9+\frac{4\bra{2\alpha-1}}{\alpha}+\frac{3-\sigma}{\alpha^2}=3\sigma-1-\frac{4}{\alpha}-\frac{\sigma-3}{\alpha^2}<3\sigma-1.\]
It follows that
\[4\sigma\alpha+3-\sigma-\nu>4\sigma\bra{\alpha-1}+4>0,\]
and therefore $\eta\geq1$. On the other hand,
\[\eta\alpha^2=4\sigma\alpha+3-\sigma-\nu\leq4\bra{2\alpha-1}\alpha-1<8\alpha^2.\]
Thus $\eta<8$, and hence $\eta\leq7$.
\smallbreak
We now define
\[\lambda=4\sigma-\eta\alpha.\]
Multiplying by $\alpha$ and using the identity for $\eta\alpha^2$ gives
\[\lambda\alpha=4\sigma\alpha-\eta\alpha^2=\sigma+\nu-3.\]
Since $\sigma\geq3$ and $\nu\geq1$, we have $\lambda\geq1$. Moreover,
\[\sigma=\frac{\eta\alpha+\lambda}{4}.\]
Substituting this into the identity for $\eta\alpha^2$ gives
\[\nu=\frac{\bra{4\lambda-\eta}\alpha+12-\lambda}{4}.\]
Using $\nu M=3\sigma-9+\eta$, we therefore obtain
\[M\left(\bra{4\lambda-\eta}\alpha+12-\lambda\right)=3\eta\alpha+3\lambda-36+4\eta.\]

Since $\nu\geq1$, we have
\[\bra{4\lambda-\eta}\alpha+12-\lambda=4\nu\geq4.\]
Since $M\geq1$, the preceding identity also gives
\[\bra{4\lambda-\eta}\alpha+12-\lambda\leq3\eta\alpha+3\lambda-36+4\eta.\]
Equivalently,
\[\bra{\eta-\lambda}\alpha+\lambda+\eta-12\geq0.\]
If $\lambda\leq\eta$, then $\lambda\leq7$. If $\lambda>\eta$, then
\[\bra{\lambda-\eta}\alpha\leq\lambda+\eta-12.\]
Since $\alpha\geq3$, this implies
\[3\bra{\lambda-\eta}\leq\lambda+\eta-12,\]
and hence
\[2\lambda\leq4\eta-12\leq16.\]
Thus, in every case,
\[1\leq\lambda\leq8.\]

The integrality of $\sigma$ and $\nu$, together with the preceding inequalities, now leaves only finitely many possibilities for $\eta$ and $\lambda$. Checking the four possible residue classes of $\alpha$ modulo $4$ for each pair with $1\leq\eta\leq7$ and $1\leq\lambda\leq8$ gives
\[
\begin{array}{c|c}
\eta&\lambda\\ \hline
3&1,2\\
5&1,2,3,4\\
6&2,4,6\\
7&2,3,4,5,6,7
\end{array}
\]

For each of these pairs, write
\[L_{\eta,\lambda}=\bra{4\lambda-\eta}\alpha+12-\lambda, \qquad R_{\eta,\lambda}=3\eta\alpha+3\lambda-36+4\eta.\]
We have
\[ML_{\eta,\lambda}=R_{\eta,\lambda},\qquad L_{\eta,\lambda}=4\nu>0.\]
In particular, $L_{\eta,\lambda}\mid R_{\eta,\lambda}$, and hence
\[L_{\eta,\lambda}\mid\bra{4\lambda-\eta}R_{\eta,\lambda}-3\eta L_{\eta,\lambda}.\]
The expression on the right is independent of $\alpha$ and is equal to
\[\Delta_{\eta,\lambda}=\bra{4\lambda-\eta}\bra{3\lambda-36+4\eta}-3\eta\bra{12-\lambda}.\]

The only listed pair for which $\Delta_{\eta,\lambda}=0$ is $\bra{\eta,\lambda}=\bra{6,6}$. In this case
\[L_{6,6}=R_{6,6}=18\alpha+6,\]
so $M=1$, which is impossible.
\smallbreak
For every other listed pair, $\Delta_{\eta,\lambda}\neq0$, and the positive integer $L_{\eta,\lambda}$ must be a divisor of $\abs{\Delta_{\eta,\lambda}}$. Since $L_{\eta,\lambda}$ is linear in $\alpha$, this leaves only finitely many possible values of $\alpha$. Enumerating these divisors and retaining only the cases in which $\alpha,M,\sigma,\nu$ are positive integers, $3\leq\sigma\leq2\alpha-1$, and $p=M\alpha^2+1$ is prime leaves
\[
\begin{array}{c|c|c|c|c|c}
\alpha&M&p&\sigma&\eta&\lambda\\ \hline
3&4&37&4&5&1\\
11&6&727&17&6&2\\
13&4&677&10&3&1
\end{array}
\]
The code for this finite divisor check is available at \cite{TyrrellCode}.
\smallbreak
It remains to verify that none of these possibilities gives two roots. If $\alpha=3$ and $\sigma=4$, the only possible pair is $\bra{r,s}=\bra{1,3}$, and
\[P\bra{3,1}\equiv P\bra{3,3}\equiv2\pmod{37}.\]
If $\alpha=11$ and $\sigma=17$, the possible pairs are
\[\bra{6,11},\qquad\bra{7,10},\qquad\bra{8,9},\]
and the corresponding common residues of $P\bra{11,r}$ and $P\bra{11,s}$ modulo $727$ are
\[190,\qquad20,\qquad662.\]
If $\alpha=13$ and $\sigma=10$, the possible pairs are
\[\bra{1,9},\qquad\bra{2,8},\qquad\bra{3,7},\qquad\bra{4,6},\]
and the corresponding common residues modulo $677$ are
\[159,\qquad315,\qquad233,\qquad590.\]
None of these residues is zero. This contradiction proves that the congruence has at most one solution in the range $1\leq x\leq\alpha$.
\end{proof}

The remainder of the section completes the proof of \thref{kdividesalpha}. Combining \thref{firstpoly,m,rootuniqueness}, we show that the power sums of $A$ and $B$ vanish at every index $1\leq j\leq\al$ which is not divisible by $k$. We then choose $S\in\{A,B\}$ such that $0\notin S$. Newton's identities imply that $e_j\bra{S}=0$ whenever $k\nmid j$. Since
\[e_\al\bra{S}=\prod_{s\in S}s\neq0,\]
it follows that $k\mid\al$.

\begin{prop}\thlabel{nonmultiplevanishing}
For every $1\leq i\leq\al$ with $k\nmid i$, we have
\[p_i\bra{A}=0.\]
\end{prop}

\begin{proof}
Suppose not. Let $m$ be the least integer such that
\[1\leq m\leq\al,\qquad k\nmid m,\qquad p_m\bra{A}\neq0.\]
By \thref{m}, we have
\[P\bra{\al,m}\equiv0\pmod p.\]
On the other hand, by \thref{firstpoly}, we have
\[P\bra{\al,k}\equiv0\pmod p.\]
By \thref{klessthanalpha}, we have $1\leq k\leq\al$, and by construction $1\leq m\leq\al$. Moreover $m\neq k$, since $k\mid k$ but $k\nmid m$. Thus the congruence
\[P\bra{\al,x}\equiv0\pmod p\]
has two distinct solutions $x=k$ and $x=m$ in the range $1\leq x\leq\al$, contradicting \thref{rootuniqueness}. Therefore no such $m$ exists, and hence
\[p_i\bra{A}=0\]
for every $1\leq i\leq\al$ with $k\nmid i$.
\end{proof}

\begin{corollary}\thlabel{Bnonmultiplevanishing}
For every $1\leq i\leq\al$ with $k\nmid i$, we have
\[p_i\bra{B}=0.\]
\end{corollary}
\begin{proof}
We prove the result by induction on $i$. Let $1\leq i\leq\al$ with $k\nmid i$, and suppose that the result has already been proved for all $i' < i$ with $k \nmid i'$.
\smallbreak
Since $i\leq\al<\al^2=\abs{H}$, we have
\[\sum_{h\in H}h^i=0.\]
Using $H=A+B$ and the fact that every element of $H$ has a unique representation as $a+b$, this gives
\[0=\sum_{a\in A,b\in B}\bra{a+b}^i=\sum_{j=0}^i\binom{i}{j}p_j\bra{A}p_{i-j}\bra{B}.\]
We now examine the terms in this sum. The term with $j=0$ is $\al p_i\bra{B}$.
The term with $j=i$ is $\al p_i\bra{A}$, which is zero by \thref{nonmultiplevanishing}, since $k\nmid i$.
\smallbreak
Now let $1\leq j\leq i-1$. If $k\nmid j$, then $p_j\bra{A}=0$ by \thref{nonmultiplevanishing}. 
\smallbreak
If $k\mid j$, then $k\nmid i-j$, since $k\nmid i$. Also $0<i-j<i$, so the induction hypothesis gives $p_{i-j}\bra{B}=0$. Thus every mixed term with $1\leq j\leq i-1$ vanishes. 
\smallbreak
Therefore
\[0=\al p_i\bra{B},\]
and since $\al<p$,
\[p_i\bra{B}=0.\]
\end{proof}

We can now prove that $k$ divides $\al$.

\begin{proof}[Proof of \thref{kdividesalpha}]
Since $H=A+B\subseteq\F_p^*$, the sets $A$ and $B$ cannot both contain $0$. Choose $S\in\{A,B\}$ such that $0\notin S$.
\smallbreak
By \thref{nonmultiplevanishing} and \thref{Bnonmultiplevanishing}, we have
\[p_i\bra{S}=0\]
for every $1\leq i\leq\al$ with $k\nmid i$. We claim that
\[e_j\bra{S}=0\]
whenever $1\leq j\leq\al$ and $k\nmid j$. We prove this by induction on $j$. Newton's identities (\thref{newton}) give
\[j e_j\bra{S}=\sum_{i=1}^{j}\bra{-1}^{i-1}e_{j-i}\bra{S}p_i\bra{S}.\]
Since $j\leq\al<p$, the integer $j$ is nonzero in $\F_p$, so it is invertible. Suppose that $k\nmid j$. In each term on the right-hand side, either $k\nmid i$, in which case $p_i\bra{S}=0$, or $k\mid i$, in which case $k\nmid j-i$. In the latter case, if $j-i>0$, the induction hypothesis gives $e_{j-i}\bra{S}=0$; the case $j-i=0$ cannot occur because it would imply $k\mid j$. Hence the whole right-hand side vanishes, and so $e_j\bra{S}=0$.
\smallbreak
Now $0\notin S$, so
\[e_\al\bra{S}=\prod_{s\in S}s\neq0.\]
By the claim, this is possible only if $k\mid\al$. Therefore $k\mid\al$, as required.
\end{proof}

\section{Uniform arithmetic obstruction}
\label{section7}

The goal of this section is to prove that $P\bra{\al,k} \not \equiv 0 \pmod{p}$ for admissible $k$ and $\al$, where $P\bra{\al,k}$ is the polynomial in \thref{firstpoly}.

\begin{prop}\thlabel{pnotdivide}
Let $k,\al$ be positive even integers such that $k \mid \al$ and $\al > 2$. Let $M \geq 2$ be an integer, and let $p = M \al^2+1$ be a prime. Then
\[p \nmid P(\al,k).\]

\end{prop}
\begin{remark}\thlabel{knot2}
    In the case $k=2$, \thref{pnotdivide} is easy. We have 
    \[P\bra{\al,2}=-12\al\bra{\al-2}.\] 
    But since $\al > 2$ and $p>\al^2$, none of $\al$, $\al-2$ or $12$ are $0$, and thus
    \[P\bra{\al,2} \not \equiv 0 \pmod{p}.\]
\end{remark}
We prove \thref{pnotdivide} by contradiction, reducing to two cases, and showing that neither can hold.

\begin{lemma}\thlabel{dichotomy}
Let $K\geq 2$, $n\geq 1$, and $M\geq 2$. Put
\[\al=2Kn,\qquad p=M\al^2+1.\]
Suppose that
\[p\mid P\bra{\al,2K}.\]
Let $\mu$ be the positive integer such that
\[P\bra{\al,2K}=\mu p,\]
and define
\[\ell:=\frac{\mu M-6\bra{2K^2-3K-1}}{2}.\]
Then $\ell$ is a positive integer,
\[\frac{\mu}{2}=4Kn\bra{4K^2+2-\ell Kn}-2K^2+3K-1,\]
and
\[\frac{\mu M}{2}=6K^2-9K-3+\ell.\]
Moreover, either
\[\ell n=4K,\]
or
\[n=1,\qquad \ell=4K-1.\]
\end{lemma}

\begin{proof}
Throughout this proof, we write
\begin{align*}C_K&=6\bra{2K^2-3K-1},\\
D_{\al,K}&=8\bra{2K^2+1}\al-2\bra{2K^2-3K+1},\end{align*}
so that $P\bra{\al,2K}=C_K \al^2 + D_{\al,K}$.
\bigbreak
Since $K\geq 2$ and $n\geq 1$, we have $P\bra{\al,2K}>0$, so $\mu$ is positive. Using $p=M\al^2+1$, the identity $P\bra{\al,2K}=\mu p$ becomes
\begin{equation}\label{eq:quotient-main}
C_K\al^2+D_{\al,K}=\mu\bra{M\al^2+1}.
\end{equation}
Since $\al$ is even, $p=M\al^2+1$ is odd. Also $P\bra{\al,2K}$ is even, so $\mu$ is even. Since $C_K$ is even, $\mu M-C_K$ is even, and hence $\ell$ is an integer.
\smallbreak
We next show that $\ell>0$. Multiplying \eqref{eq:quotient-main} by $M$ and using $M\al^2=p-1$, we get
\[C_K\bra{p-1}+MD_{\al,K}=\mu Mp.\]
Rearranging gives
\[MD_{\al,K}-C_K=\bra{\mu M-C_K}p=2\ell p.\]
Moreover,
\[MD_{\al,K}-C_K\geq D_{\al,K}-C_K=4\bra{8K^3n-4K^2+4Kn+6K+1}>0.\]
Thus $\ell>0$.
\smallbreak
By the definition of $\ell$, we have
\[\mu M=C_K+2\ell.\]
Substituting this into \eqref{eq:quotient-main} gives
\[C_K\al^2+D_{\al,K}=\bra{C_K+2\ell}\al^2+\mu.\]
Hence
\[\mu=D_{\al,K}-2\ell\al^2.\]
Dividing the previous identity by $2$ and using $\al=2Kn$ gives
\begin{align}
\frac{\mu}{2}&=\frac{D_{\al,K}}{2}-\ell\al^2\\
&=4Kn\bra{4K^2+2-\ell Kn}-2K^2+3K-1. \label{eq:R-formula}
\end{align}
Since $\mu>0$, this proves
\[4Kn\bra{4K^2+2-\ell Kn}-2K^2+3K-1>0.\]
Dividing $\mu M=C_K+2\ell$ by $2$ gives
\[\frac{\mu M}{2}=6K^2-9K-3+\ell.\]
Using \eqref{eq:R-formula}, this is exactly
\[M\bra{4Kn\bra{4K^2+2-\ell Kn}-2K^2+3K-1}=6K^2-9K-3+\ell.\]
\smallbreak
We now show that the two identities just obtained force either $\ell n=4K$ or $n=1$ and $\ell=4K-1$. Since $M \geq 2$ and
\[\frac{\mu}{2}M=6K^2-9K-3+\ell,\]
we have
\begin{equation}\label{eq:R-bound}
0<\frac{\mu}{2}\leq 6K^2-9K-3+\ell.
\end{equation}
Define
\[e=4K-\ell n.\]
Then
\[4K^2+2-\ell Kn=2+Ke,\]
so \eqref{eq:R-formula} becomes
\begin{equation}\label{eq:R-e-formula}
\frac{\mu}{2}=\bra{4en-2}K^2+\bra{8n+3}K-1.
\end{equation}
If $e<0$, then $2+Ke\leq 2-K\leq 0$, and hence \eqref{eq:R-formula} gives
\[\frac{\mu}{2}\leq -2K^2+3K-1<0,\]
contradicting $\mu>0$. Therefore $e\geq 0$.
\smallbreak
Suppose first that $e\geq 1$ and $n\geq 2$. Since
\[\ell =\frac{4K-e}{n},\]
we have $\ell \leq 2K-1$. Hence
\[6K^2-9K-3+\ell \leq 6K^2-7K-4.\]
On the other hand, \eqref{eq:R-e-formula} gives
\[\frac{\mu}{2}\geq 6K^2+19K-1.\]
This contradicts \eqref{eq:R-bound}.
\smallbreak
Suppose next that $e\geq 2$ and $n=1$. Then $\ell =4K-e\leq 4K-2$, so
\[6K^2-9K-3+\ell \leq 6K^2-5K-5.\]
But \eqref{eq:R-e-formula} gives
\[\frac{\mu}{2}\geq 6K^2+11K-1,\]
again contradicting \eqref{eq:R-bound}.
\smallbreak
Thus the only remaining possibilities are $e=0$, or $e=1$ and $n=1$. Since $e=4K-\ell n$, these are exactly
\[\ell n=4K\]
and
\[n=1,\qquad \ell =4K-1.\]
\end{proof}

We now rule out the two options in \thref{dichotomy} in turn.

\begin{lemma}\thlabel{case1}
The case
\[n=1,\qquad \ell =4K-1\]
in \thref{dichotomy} is impossible.
\end{lemma}

\begin{proof}
Assume that
\[n=1,\qquad \ell =4K-1.\]

Substituting $n=1$ and $\ell =4K-1$ gives
\[\frac{\mu}{2}=2K^2+11K-1\]
and by \thref{dichotomy}
\[\frac{\mu M}{2}=6K^2-9K-3+\ell =6K^2-5K-4.\]
Thus
\[\frac{\mu}{2}\mid 6K^2-5K-4.\]
For $K=2,3,4$, one checks directly that
\[\frac{\mu}{2}>6K^2-5K-4,\]
contradicting the divisibility. We may therefore assume $K\geq5$.
\smallbreak
We also have
\[6K^2-5K-4<3\bra{2K^2+11K-1},\]
and hence $M<3$. Since $M\geq2$, it follows that $M=2$. Therefore
\[2\bra{2K^2+11K-1}=6K^2-5K-4.\]
Reducing this equality modulo $K$ gives
\[-2\equiv-4\pmod K,\]
and hence $K\mid2$, contradicting $K\geq5$.
\end{proof}

\begin{lemma}\thlabel{case2part1}
Suppose that
\[\ell n=4K\]
in \thref{dichotomy}. Then
\[5\leq \ell \leq20.\]
\end{lemma}

\begin{proof}
Assume that \thref{dichotomy} holds with
\[\ell n=4K,\]
and so
\[\frac{\mu}{2}=\bra{\frac{32}{\ell }-2}K^2+3K-1.\]
By \thref{dichotomy}, $\frac{\mu}{2}\leq6K^2-9K-3+\ell $, and hence
\[\bra{\frac{32}{\ell }-2}K^2+3K-1\leq6K^2-9K-3+\ell.\]
Multiplying by $\ell $ and rearranging gives
\[0\leq8\bra{\ell -4}K^2-12\ell K+\ell \bra{\ell -2}.\]
If $\ell \leq4$, the right-hand side is negative for every $K\geq2$, a contradiction. Hence $\ell \geq5$.
\smallbreak
It remains to prove the upper bound. If $K\leq5$, then $\ell \mid4K$ gives
\[\ell \leq4K\leq20.\]
If $K\geq6$ and $\ell \geq21$, then
\[\frac{\mu}{2}=\bra{\frac{32}{\ell }-2}K^2+3K-1\leq-\frac{10}{21}K^2+3K-1<0,\]
contradicting $\mu>0$. Therefore $\ell \leq20$.
\end{proof}

\begin{lemma}\thlabel{case2part2}
The case
\[\ell n=4K\]
in \thref{dichotomy} is impossible.
\end{lemma}

\begin{proof}
Assume that \thref{dichotomy} holds with
\[\ell n=4K.\]
By \thref{case2part1}, we have
\[5\leq\ell\leq20.\]
Moreover, $\ell\mid4K$, and \thref{dichotomy} gives
\[\frac{\mu}{2}=\bra{\frac{32}{\ell}-2}K^2+3K-1\]
and
\[\frac{\mu M}{2}=6K^2-9K-3+\ell.\]
Equivalently,
\[\ell\frac{\mu}{2}=\bra{32-2\ell}K^2+3\ell K-\ell.\]

First suppose that
\[5\leq\ell\leq15.\]
Multiplying the preceding identity by $M$ and using the formula for $\mu M/2$, we obtain
\[M\bra{\bra{32-2\ell}K^2+3\ell K-\ell}=\ell\bra{6K^2-9K-3+\ell}.\]
Rearranging gives
\[\bra{2M\bra{16-\ell}-6\ell}K^2+3\ell\bra{M+3}K+\ell\bra{3-\ell-M}=0.\]

We next bound $M$. Using the two identities involving $\mu/2$ and $\mu M/2$, we have
\[\bra{6\ell-\bra{32-2\ell}M}\frac{\mu}{2} =6\left(\ell\frac{\mu}{2}\right)-\bra{32-2\ell}\frac{\mu M}{2}\]
\[=288K+2\bra{\ell^2-22\ell+48}.\]
For $5\leq\ell\leq15$ and $K\geq2$, the right-hand side is positive. Since $\mu/2>0$, it follows that
\[6\ell-\bra{32-2\ell}M>0.\]
Therefore
\[1\leq M<\frac{3\ell}{16-\ell}.\]

There are consequently only finitely many possible pairs $\bra{\ell,M}$. For each pair satisfying
\[5\leq\ell\leq15,\qquad 1\leq M<\frac{3\ell}{16-\ell},\]
the equation
\[\bra{2M\bra{16-\ell}-6\ell}K^2+3\ell\bra{M+3}K+\ell\bra{3-\ell-M}=0\]
is a quadratic equation in $K$. Its discriminant is
\[\Delta_{\ell,M}=9\ell^2\bra{M+3}^2-4\bra{2M\bra{16-\ell}-6\ell}\ell\bra{3-\ell-M}.\]
Since $K$ is an integer, $\Delta_{\ell,M}$ must be a square. Checking the finite range above, the only cases in which $\Delta_{\ell,M}$ is a square and the quadratic has a positive rational root are
\[
\begin{array}{c|c|c}
\ell&M&\text{positive roots }K\\ \hline
8&2&\frac12,\ 7\\
14&3&\frac76,\ \frac73\\
14&14&\frac12,\ 25
\end{array}
\]
None of these gives an integer $K$ satisfying $\ell\mid4K$. For $\bra{\ell,M}=\bra{8,2}$, the only positive integer root is $K=7$, but $8\nmid28$. For $\bra{\ell,M}=\bra{14,3}$, neither positive root is an integer. For $\bra{\ell,M}=\bra{14,14}$, the only positive integer root is $K=25$, but $14\nmid100$. Hence
\[5\leq\ell\leq15\]
is impossible. The code for this finite check is available at \cite{TyrrellCode}.

Now suppose that
\[\ell=16.\]
Then
\[\frac{\mu}{2}=3K-1.\]
Since $\ell\mid4K$, we have $4\mid K$. Moreover,
\[\frac{\mu M}{2}=6K^2-9K+13,\]
so
\[3K-1\mid6K^2-9K+13.\]
Consequently,
\[3K-1\mid3\bra{6K^2-9K+13}.\]
But
\[3\bra{6K^2-9K+13}=\bra{6K-7}\bra{3K-1}+32,\]
and hence
\[3K-1\mid32.\]
Since $4\mid K$ and $K\geq2$, we have $K\geq4$, so $3K-1\geq11$. The only positive divisors of $32$ which are at least $11$ are $16$ and $32$, and neither is of the form $3K-1$ with $4\mid K$. Therefore $\ell=16$ is impossible.
\smallbreak
It remains to consider
\[17\leq\ell\leq20.\]
Since $\mu/2>0$, we have
\[\bra{32-2\ell}K^2+3\ell K-\ell>0.\]
Thus
\[\bra{2\ell-32}K^2<3\ell K-\ell<3\ell K,\]
and hence
\[K<\frac{3\ell}{2\ell-32}.\]
Combining this bound with $\ell\mid4K$ leaves only the following possibilities:
\[
\begin{array}{c|c|c|c}
\ell&K&\frac{\mu}{2}&\bra{6K^2-9K-3+\ell}\bmod\frac{\mu}{2}\\ \hline
17&17&16&11\\
18&9&8&4\\
20&5&4&2
\end{array}
\]
In each case,
\[\frac{\mu}{2}\nmid6K^2-9K-3+\ell,\]
contradicting the identity
\[\frac{\mu M}{2}=6K^2-9K-3+\ell.\]
Therefore the case $\ell n=4K$ is impossible.
\end{proof}

\begin{proof}[Proof of \thref{pnotdivide}.]
The case $k=2$ follows from \thref{knot2}. We may therefore write
\[k=2K\]
for some integer $K\geq2$. Since $k\mid\al$, there is an integer $n\geq1$ such that
\[\al=2Kn.\]
Suppose for a contradiction that
\[p\mid P\bra{\al,k}.\]
By \thref{dichotomy}, either
\[n=1,\qquad \ell=4K-1,\]
or
\[\ell n=4K.\]
The first alternative is excluded by \thref{case1}, and the second is excluded by \thref{case2part2}. This contradiction proves that
\[p\nmid P\bra{\al,k},\]
as required.
\end{proof}

We are now ready to prove our main result.
\begin{proof}[Proof of \thref{main}.]
Let $H \leq \F_p^*$ be a proper multiplicative subgroup of $\F_p$, and assume that we can write
\[H = A + B,\]
where $\abs{A}, \abs{B} > 2$. If we let $\abs{A}=\al$, then it follows from \thref{kalmyninsize} that $p = M \al^2 +1 $ for some positive integer $M \geq 2$. Translating $A$ and $B$ if necessary so that $p_1\bra{A} = 0$, let $k \geq 2$ be the least positive integer such that
\[\sum_{a \in A}a^k \neq 0.\]
By \thref{keven} we know $k$ is even, and from \thref{kdividesalpha} that $k \mid \al$, which implies $\al$ is also even. Thus, by \thref{pnotdivide}, we have $P\bra{\al,k}\not \equiv 0 \pmod p$, for $\al >2$. But by \thref{firstpoly}, $\al,k$ satisfy $P\bra{\al,k}\equiv0 \pmod p$, which is a contradiction, so we cannot write $H=A+B$ with $\abs{A},\abs{B} >2$.
\smallbreak
The case when $\abs{A},\abs{B}$ are not both greater than $2$ is dealt with in \thref{trivial}, giving exactly the two cases in the statement of \thref{main}.
 
\end{proof}

\appendix

\section{Combinatorial identities}

In this appendix, we prove \thref{generalsums}, and then deduce the explicit formula in \thref{coeffs2} for the Hanson-Petridis coefficients and the complete homogeneous polynomial identity in \thref{homo}. \thref{homo} is known in the literature as Sylvester's identity - see \cite{nica2023identity} for several different proofs, as well as a discussion.

\begin{proof}[Proof of \thref{generalsums}.]
We first prove the first identity by induction on $t$. When $t=0$, we necessarily have $k\geq0$, and
\[S_{k,0}\bra{B}=\sum_{b_0\in B}b_0^k,\]
which is precisely the required identity.
\smallbreak
Now suppose that $t\geq1$, and assume that the result holds with $t-1$ in place of $t$. Let $q=k+t$, so that $q\geq0$. For pairwise distinct $b_0,\ldots,b_t$, the partial-fraction identity
\[\frac{1}{\prod_{i=1}^t\bra{b_0-b_i}}=\sum_{i=1}^t\frac{1}{b_0-b_i}\frac{1}{\prod_{j\in[t]\setminus\{i\}}\bra{b_i-b_j}}\]
holds. After summing over all pairwise distinct $\bra{t+1}$-tuples, each of the $t$ terms on the right gives the same contribution. Therefore
\[S_{k,t}\bra{B}=t\sum^*_{B^{t+1}}\frac{b_0^q}{\bra{b_0-b_1}\prod_{i=2}^t\bra{b_1-b_i}}.\]

Writing
\[b_0^q=\bra{b_0^q-b_1^q}+b_1^q\]
and swapping $b_0$ and $b_1$ in the sum arising from the second term, we obtain
\[(t+1)S_{k,t}\bra{B}=t\sum^*_{B^{t+1}}\frac{b_0^q-b_1^q}{\bra{b_0-b_1}\prod_{i=2}^t\bra{b_1-b_i}}.\]
When $q=0$, we have $k=-t<0$, so both sides of the first identity are zero. We may therefore assume that $q\geq1$. Using
\[\frac{b_0^q-b_1^q}{b_0-b_1}=\sum_{r=0}^{q-1}b_0^r b_1^{q-1-r},\]
we obtain
\[(t+1)S_{k,t}\bra{B}=t\sum_{r=0}^{q-1}\sum_{b_0\in B}b_0^r\sum^*_{\bra{B\setminus\{b_0\}}^t}\frac{b_1^{q-1-r}}{\prod_{i=2}^t\bra{b_1-b_i}}.\]
The innermost sum is
\[S_{k-r,t-1}\bra{B\setminus\{b_0\}},\]
since
\[\bra{k-r}+\bra{t-1}=q-1-r.\]
Moreover, because $0\leq r\leq q-1$, we have
\[k-r\geq1-t=-(t-1),\]
so the inductive hypothesis applies. It gives
\[tS_{k-r,t-1}\bra{B\setminus\{b_0\}}=\sum_{\substack{i_1+\cdots+i_t=k-r\\i_1,\ldots,i_t\geq0}}\sum^*_{\bra{B\setminus\{b_0\}}^t}
b_1^{i_1}\cdots b_t^{i_t},\]
where the right-hand side is zero when $k-r<0$.
\smallbreak
Substituting this identity, we obtain
\[(t+1)S_{k,t}\bra{B}=\sum_{r=0}^{q-1}\sum_{b_0\in B}b_0^r\sum_{\substack{i_1+\cdots+i_t=k-r\\i_1,\ldots,i_t\geq0}}
\sum^*_{\bra{B\setminus\{b_0\}}^t}b_1^{i_1}\cdots b_t^{i_t}\]
\[=\sum_{\substack{i_0+\cdots+i_t=k\\i_0,\ldots,i_t\geq0}}\sum^*_{B^{t+1}}b_0^{i_0}\cdots b_t^{i_t}.\]
Indeed, the terms with $r>k$ vanish, while the remaining terms correspond exactly to the choices $i_0=r$ in a composition of $k$. This proves the first identity, including the case $k<0$.
\smallbreak
For the second identity, group the ordered tuples according to their underlying unordered sets. For each set 
\[C=\{b_0,\ldots,b_t\}\sub B\]
of size $t+1$, there are $(t+1)!$ possible orderings, and for each ordering
\[\sum_{\substack{i_0+\cdots+i_t=k\\i_0,\ldots,i_t\geq0}} b_0^{i_0}\cdots b_t^{i_t}=h_k\bra{C}.\]
The value is independent of the chosen ordering. Hence
\[\sum_{\substack{i_0+\cdots+i_t=k\\i_0,\ldots,i_t\geq0}}\sum^*_{B^{t+1}}b_0^{i_0}\cdots b_t^{i_t}=(t+1)!\sum_{\substack{C\sub B\\\abs{C}=t+1}}h_k\bra{C},\]
as required.
\end{proof}

We now apply \thref{generalsums} when the number of variables is the cardinality of the underlying set, to prove the following, from which we can easily deduce \thref{coeffs2,homo}.

\begin{lemma}\thlabel{interpolationcoefficients}
Let $S\subsetneq\F_p$ have size $r$, and for each $s\in S$ define
\[\lambda_S\bra{s}:=\frac{1}{\prod_{t\in S\setminus\{s\}}\bra{s-t}}.\]
Then, for every integer $j\geq0$,
\[\sum_{s\in S}s^j\lambda_S\bra{s}=
\begin{cases}
0,&0\leq j\leq r-2,\\
h_{j-r+1}\bra{S},&j\geq r-1.
\end{cases}\]
\end{lemma}

\begin{proof}
Apply \thref{generalsums} with
\[B=S,\qquad t=r-1,\qquad k=j-r+1.\]
Since $j\geq0$, we have
\[k\geq1-r=-t,\]
so the lemma applies.
\smallbreak
For a fixed $s\in S$, there are $(r-1)!$ ways to order the remaining elements of $S$. Hence
\[S_{j-r+1,r-1}\bra{S}=(r-1)!\sum_{s\in S}\frac{s^j}{\prod_{t\in S\setminus\{s\}}\bra{s-t}}=(r-1)!\sum_{s\in S}s^j\lambda_S\bra{s}.\]
If $0\leq j\leq r-2$, then $k<0$, and \thref{generalsums} gives
\[rS_{j-r+1,r-1}\bra{S}=0.\]
Since $r!$ is nonzero in $\F_p$, it follows that
\[\sum_{s\in S}s^j\lambda_S\bra{s}=0.\]

If $j\geq r-1$, then $k\geq0$. Since $S$ is the only subset of itself having cardinality $r$, \thref{generalsums} gives
\[rS_{j-r+1,r-1}\bra{S}=r!h_{j-r+1}\bra{S}.\]
Using the preceding expression for $S_{j-r+1,r-1}\bra{S}$ and cancelling $r!$, we obtain
\[\sum_{s\in S}s^j\lambda_S\bra{s}=h_{j-r+1}\bra{S}.\]
This proves the result.

\end{proof}

\thref{coeffs2,homo} are now almost immediate corollaries of \thref{interpolationcoefficients}.

\begin{proof}[Proof of \thref{coeffs2}.]
Let $\abs{S}=r$. By \thref{interpolationcoefficients}, the coefficients $\lambda_S\bra{s}$ satisfy
\[\sum_{s\in S}s^j\lambda_S\bra{s}=
\begin{cases}
0,&0\leq j\leq r-2,\\
1,&j=r-1.
\end{cases}\]
The coefficients $c_S\bra{s}$ are defined as the unique solution to this system. Uniqueness follows from the invertibility of the
Vandermonde matrix associated with the distinct elements of $S$. Therefore
\[c_S\bra{s}=\lambda_S\bra{s}
=\frac{1}{\prod_{t\in S\setminus\{s\}}\bra{s-t}},\]
as required.
\end{proof}

\begin{proof}[Proof of \thref{homo}.]
Let $S\sub\F_p$ have cardinality $r<p$, and let $m\geq0$. Applying \thref{interpolationcoefficients} with $j=m+r-1$ gives
\[\sum_{s\in S}c_S\bra{s}s^{m+r-1}=h_m\bra{S},\]
which is the desired identity.
\end{proof}

\printbibliography
\end{document}